\documentclass[reqno, 12pt]{amsart}%
\usepackage{amsmath}
\usepackage{amssymb}
\usepackage{amsfonts}
\usepackage{graphicx}%
\providecommand{\U}[1]{\protect\rule{.1in}{.1in}}
\newtheorem{theorem}{Theorem}
\theoremstyle{plain}
\newtheorem{acknowledgement}{Acknowledgement}

\newtheorem{definition}{Definition}

\newtheorem{lemma}{Lemma}

\newtheorem{problem}{Problem}
\newtheorem{proposition}{Proposition}
\newtheorem{remark}{Remark}

\numberwithin{equation}{section}
\numberwithin{theorem}{section}
\numberwithin{proposition}{section}
\numberwithin{remark}{section}
\numberwithin{definition}{section}
\numberwithin{lemma}{section}
\numberwithin{corollary}{section}
\numberwithin{example}{section}
\numberwithin{claim}{section}
\begin{document}
\title[Multiscale analysis of wave equations]{Multiscale analysis of semilinear damped stochastic wave equations}
\author{Aurelien Fouetio}
\address{A. Fouetio, Department of Mathematics and Computer Science, University of
Dschang, P.O. Box 67, Dschang, Cameroon}
\email{aurelienfouetio@yahoo.fr}
\author{Gabriel Nguetseng}
\address{G. Nguetseng, Department of Mathematics, University of Yaounde I, P.O. Box
812, Yaounde, Cameroon}
\email{nguetsengg@yahoo.fr}
\author{Jean Louis Woukeng}
\address{J.L. Woukeng, Department of Mathematics and Computer Science, University of
Dschang, P.O. Box 67, Dschang, Cameroon}
\email{jwoukeng@yahoo.fr}
\date{}

\begin{abstract}
In this paper we proceed with the multiscale analysis of semilinear damped
stochastic wave motions. The analysis is made by combining the well-known
sigma convergence method with its stochastic counterpart, associated to some
compactness results such as the Prokhorov and Skorokhod theorems. We derive
the equivalent model, which is of the same type as the micro-model.

\end{abstract}
\maketitle

\section{Introduction}

An important class of physical phenomena is the one represented by wave
phenomena. In continuous media, the most frequent form of energy transmission
is through waves. As clearly known, wave nature is apparent in very important
phenomena such as the sound propagation, the electromagnetic radiation's
propagation, just to cite a few. It appears in many fields of science such as
the acoustics, optics, geophysics, radiophysics and mechanics. However, in
order to take into account some physical data such as turbulence, the most
accurate models are known to be stochastic ones. Indeed, damping may occur
through wave scattering due to the inhomogeneity of the medium and the surface
(size of inhomogeneity in relation to wavelength, inhomogeneities distributed
in a continuous manner or in the form of a random/deterministic discrete
scattering elements). Hence, because of damping, scattered waves may cause
fluctuations of the amplitude, leading to the justification of the use of
stochastic models. In the quest of mastering wave motion, it is very important
to study the phenomenon starting from a micro-model (like in (\ref{eqnn1.1})
below) in order to predict the overall behavior on the large scale model.
Hence the importance of multiscale analysis in such a study.

To be more precise, we study waves motion through an inhomogeneous medium made
of microstructures of small size (representing the inhomogeneities). Because
of what we said above, we assume that the microstructures are
deterministically distributed in the medium, and their distribution is
represented by an assumption made on the fast spatial variable
$y=x/\varepsilon$ covering a wide range of behaviors such as the uniform (or
periodic) distribution, the almost periodic distribution and the asymptotic
almost periodic one. The study also takes into account an assumption made on
the fast time variable $\tau=t/\varepsilon$,\ which allows us to predict the
long time wave behavior. The micro-model is presented in the following lines.

Let $Q$ be a Lipschitz domain of $\mathbb{R}^{N}$, $T$ and $\varepsilon$ be
positive real numbers with $0<\varepsilon<1$ , $Q_{T}=Q\times\left(
0,T\right)  $ and $W=\left(  W\left(  t\right)  \right)  _{0\leq t\leq T}$ a
$m$-dimensional standard Wiener process defined on a given probability space
$\left(  \Omega,\mathcal{F},\mathbb{P}\right)  $ equipped with the filtration
$\left(  \mathcal{F}_{t}\right)  _{0\leq t\leq T}$. We consider the elliptic
linear partial differential operator%
\[
P^{\varepsilon}=-\sum_{i,j=1}^{N}\frac{\partial}{\partial x_{i}}\left(
a_{ij}^{\varepsilon}(x)\frac{\partial}{\partial x_{j}}\right)  \text{.}%
\]
We study the asymptotic behavior (as $0<\varepsilon\rightarrow0$) of the
sequence of solutions $u_{\varepsilon}$ of the following initial-boundary
value problem
\begin{equation}
\left\{
\begin{array}
[c]{l}%
du_{\varepsilon}^{\prime}+P^{\varepsilon}u_{\varepsilon}dt-\Delta
u_{\varepsilon}^{\prime}dt=f(\frac{x}{\varepsilon},\frac{t}{\varepsilon
},u_{\varepsilon}^{\prime})dt+g(\frac{x}{\varepsilon},\frac{t}{\varepsilon
},u_{\varepsilon}^{\prime})dW\text{ \ in }Q_{T}\\
u_{\varepsilon}=0\text{ on }\partial Q\times\left(  0,T\right)  \text{
\ \ \ \ \ \ \ \ \ \ \ \ \ \ \ \ \ \ \ \ \ \ \ \ \ \ \ \ \ \ \ \ \ \ \ \ \ \ \ \ \ \ \ \ \ \ \ \ \ \ \ \ \ \ \ \ \ \ \ \ \ }%
\\
u_{\varepsilon}\left(  x,0\right)  =u^{0}\left(  x\right)  \text{ and
}u_{\varepsilon}^{\prime}\left(  x,0\right)  =u^{1}\left(  x\right)  \text{ in
}Q\text{,\ \ \ \ \ \ \ \ \ \ \ \ \ \ \ \ \ \ \ \ \ \ \ \ \ \ \ \ \ \ \ }%
\end{array}
\right.  \label{eqnn1.1}%
\end{equation}
where $u_{\varepsilon}^{\prime}=\frac{\partial u_{\varepsilon}}{\partial t}$.
We impose on the coefficients of (\ref{eqnn1.1}) the following constraints:

\begin{itemize}
\item[(\textbf{A1})] The matrix $a^{\varepsilon}=(a_{ij}^{\varepsilon})_{1\leq
i,j\leq N}$ is defined by $a_{ij}^{\varepsilon}\left(  x\right)
=a_{ij}\left(  x,\frac{x}{\varepsilon}\right)  $ ($x\in Q$), where
\end{itemize}

\begin{enumerate}
\item[(1)] $a_{ij}\in\mathcal{C}(\overline{Q},L^{\infty}(\mathbb{R}_{y}%
^{N}))(1\leq i,j\leq N)$ with $a_{ij}=a_{ji}$,

\item[(2)] there exists a constant $\alpha>0$ such that
\[
\sum_{i,j=1}^{N}a_{ij}(x,y)\xi_{i}\xi_{j}\geq\alpha\left\vert \xi\right\vert
^{2}\text{,}%
\]
\ 
\end{enumerate}

for all $\xi\in\mathbb{R}^{N},\ x\in\overline{Q}$ and a.e. $y\in\mathbb{R}%
^{N}$.

\begin{itemize}
\item[(\textbf{A2})] We define the maps $f$, $g$ and the functions $u^{0}$,
$u^{1}$ such that:

\item[(1)] The function $f:(y,\tau,v)\longmapsto f(y,\tau,v)$ from
$\mathbb{R}^{N}\times\mathbb{R}\times\mathbb{R}$ into $\mathbb{R}$ satisfies
the properties:

\begin{enumerate}
\item[(i)] $f$ is measurable,

\item[(ii)] there exists a constant $c_{1}>0$ such that $\left\vert
f(y,\tau,v)\right\vert ^{2}\leq c_{1}(1+\left\vert v\right\vert ^{2})$, for
almost all ($y,\tau)\in\mathbb{R}^{N}\times\mathbb{R}$ and for all
$v\in\mathbb{R}$.

\item[(iii)] there exists a constant $c_{2}>0$ such that%
\[
\left\vert f(y,\tau,v_{1})-f(y,\tau,v_{2})\right\vert \leq c_{2}\left\vert
v_{1}-v_{2}\right\vert
\]
for almost \ all ($y,\tau)\in\mathbb{R}^{N}\times\mathbb{R}$ and for all
$v_{1},v_{2}\in\mathbb{R}$.
\end{enumerate}

\item[(2)] The function $g:(y,\tau,v)\longmapsto g(y,\tau,v)$ from
$\mathbb{R}^{N}\times\mathbb{R}\times\mathbb{R}$ into $\mathbb{R}^{m}$
satisfies the properties:

\begin{enumerate}
\item[(i)] $g$ is measurable,

\item[(ii)] there exists a constant $c_{3}>0$ such that $\left\vert
g(y,\tau,v)\right\vert ^{2}\leq c_{3}(1+\left\vert v\right\vert ^{2})$, for
a.e. ($y,\tau)\in\mathbb{R}^{N}\times\mathbb{R}$ and for all $v\in\mathbb{R}$.

\item[(iii)] there exists a constant $c_{4}>0$ such that%
\[
\left\vert g(y,\tau,v_{1})-g(y,\tau,v_{2})\right\vert \leq c_{4}\left\vert
v_{1}-v_{2}\right\vert
\]
for a.e. ($y,\tau)\in\mathbb{R}^{N}\times\mathbb{R}$ and for all $v_{1}%
,v_{2}\in\mathbb{R}$.
\end{enumerate}

\item[(3)] $u^{0}\in H_{0}^{1}\left(  Q\right)  $ and $u^{1}\in L^{2}\left(
Q\right)  .$
\end{itemize}

Under the above conditions, it is easily seen that if $u_{\varepsilon}$ and
$v_{\varepsilon}$ are two solutions to (\ref{eqnn1.1})\ on the same stochastic
system with the same initial condition then $u_{\varepsilon}(t)=v_{\varepsilon
}(t)$ in $H_{0}^{1}\left(  Q\right)  $ almost surely for any $t$. Thanks to
this fact together with Yamada-Wanatabe's Theorem \cite{9} and the existence
Theorem in \cite[Theorem 2.2]{10} , we see that problem (\ref{eqnn1.1}) (for
any fixed $\varepsilon>0$) possesses a unique strong solution $u_{\varepsilon
}$ satisfying:

\begin{itemize}
\item[(i)] $u_{\varepsilon}$ and $u_{\varepsilon}^{\prime}$ are continuous
with respect to time in $H_{0}^{1}\left(  Q\right)  $ and $L^{2}\left(
Q\right)  $ respectively,

\item[(ii)] $u_{\varepsilon}$ and $u_{\varepsilon}^{\prime}$ are
$\mathcal{F}_{t}$-measurable,

\item[(iii)] $u_{\varepsilon}\in L^{2}(\Omega,\mathcal{F},\mathbb{P}%
;L^{\infty}(0,T;H_{0}^{1}(Q)))$, $u_{\varepsilon}^{\prime}\in L^{2}%
(\Omega,\mathcal{F},\mathbb{P};L^{\infty}(0,T;L^{2}(Q))\cap L^{2}%
(0,T;H_{0}^{1}(Q)))$,

\item[(iv)] $u_{\varepsilon}$ satisfies
\begin{align*}
&  \left(  u_{\varepsilon}^{\prime}(t),\phi\right)  +\int_{0}^{t}\left[
\left(  P^{\varepsilon}u_{\varepsilon}\left(  \tau\right)  ,\phi\right)
-\left(  \Delta u_{\varepsilon}^{\prime}\left(  \tau\right)  ,\phi\right)
\right]  d\tau\\
&  =\left(  u^{1},\phi\right)  +\int_{0}^{t}\left(  f\left(  \frac
{x}{\varepsilon},\frac{\tau}{\varepsilon},u_{\varepsilon}^{\prime}%
(\tau)\right)  ,\phi\right)  d\tau+\left(  \int_{0}^{t}g\left(  \frac
{x}{\varepsilon},\frac{\tau}{\varepsilon},u_{\varepsilon}^{\prime}%
(\tau)\right)  dW\left(  \tau\right)  ,\phi\right)  \\
\text{for all }\phi &  \in H_{0}^{1}\left(  Q\right)  \text{ and for a.e.
}t\in\left(  0,T\right)  \text{.}%
\end{align*}

\item[(v$)$] $u_{\varepsilon}(0)=u^{0}$.
\end{itemize}

Moreover the solution $u_{\varepsilon}$ satisfies the following a priori
estimates:%
\begin{equation}
\mathbb{E}\sup_{0\leq t\leq T}\left\Vert u_{\varepsilon}(t)\right\Vert
_{H_{0}^{1}\left(  Q\right)  }^{4}+\mathbb{E}\sup_{0\leq t\leq T}\left\Vert
u_{\varepsilon}^{\prime}(t)\right\Vert _{L^{2}\left(  Q\right)  }^{4}\leq
C\text{.} \label{eqe}%
\end{equation}%
\begin{equation}
\text{ }\mathbb{E}\sup_{\left\vert \theta\right\vert \leq\delta}\int_{0}%
^{T}\left\Vert u_{\varepsilon}^{\prime}(t+\theta)-u_{\varepsilon}^{\prime
}(t)\right\Vert _{H^{-1}(Q)}^{2}dt\leq C\delta\text{.} \label{eqee}%
\end{equation}
where $\delta$ is sufficiently small and $C$ is a positive constant.

The estimates (\ref{eqe})-(\ref{eqee}) together with some other estimates, are
proved in Appendix A1 at the end of the paper.

Let us consider the space $S=\mathcal{C}(0,T;\mathbb{R}^{m})\times
(L^{2}(0,T;L^{2}(Q))\cap\mathcal{C}(0,T;H^{-1}(Q)))\times\mathcal{C}%
(0,T;L^{2}(Q))$, a metric space equipped with its Borel $\sigma$-algebra
$\mathcal{B(}S)$. For $0<\varepsilon<1$, $\Psi_{\varepsilon}$ be the
measurable $S$-valued mapping defined on\ $(\Omega,\mathcal{F},\mathbb{P}%
)$\ as
\[
\Psi_{\varepsilon}(\omega)=(W(\omega,\cdot),u_{\varepsilon}(\omega
,\cdot),u_{\varepsilon}^{\prime}(\omega,\cdot))
\]
where $u_{\varepsilon}$ is the solution of problem (\ref{eqnn1.1}). We
introduce the image of $\mathbb{P}$ under $\Psi_{\varepsilon}$ defined by
\begin{equation}
\pi^{\varepsilon}(B)=\mathbb{P}(\Psi_{\varepsilon}^{-1}(B))\text{ for any
}B\in\mathcal{B}(S). \label{eqqq}%
\end{equation}

We easily prove as in \cite[Lemma 7]{7'} that the family of probability
measures $(\pi^{\varepsilon})$ constructed above is tight on $(S,\mathcal{B}%
(S))$.

One of the motivation of this work lies on the presence of the damping term
$-\Delta u_{\varepsilon}^{\prime}$. Indeed as seen in Lemma \ref{l2.1}, it
allows to choose the corrector term to be independent of the fast time
variable, although the micro-problem (\ref{eqnn1.1}) depends on the fast time
variable $\tau=t/\varepsilon$. This simplifies the homogenization process as
seen in the passage to the limit step.

The homogenization of hyperbolic stochastic partial differential equations
with rapidly oscillating coefficients has been considered for the first time
by Mohammed and Sango in a series of papers \cite{7', MS2016-1, MS2016-2,
M2017}. These works deal with linear as well as nonlinear type equations and
assume only the periodicity on their coefficients. Also they do not consider
oscillations in time $\tau=t/\varepsilon$, which is one of the main features
of the work under consideration here. Because of the oscillation in time, the
limit passage in the stochastic term is more involved as seen in the proof of
Lemma \ref{l3.2}.

Since we consider a general deterministic behavior with respect to the
oscillations in space and time, our study therefore falls within the scope of
the \textit{sigma-convergence }concept introduced in 2003 in \cite{Hom1} by
the second author. In order to get the homogenized limit, we combine two
different concepts of sigma-convergence: the sigma-convergence for stochastic
processes (see Definition \ref{d2.1}), which helps to pass to the limit in the
stochastic term, and the usual sigma-convergence method, which is used to
study the terms not involving the stochastic integral; see the proof of
Proposition \ref{p2.1}. It is worth noticing that the linear counterpart of
problem (\ref{eqnn1.1}) has been considered in \cite{FW2017}.

The layout of the paper is as follows. In section 2, we present an overview of
the fundamentals of the sigma-convergence method, which has been introduced in
\cite{Hom1}. Section 3 deals with the homogenization result. In Section 4, we
provide some applications of the homogenization result. We end the work with
two appendices in which we provide some key estimates of the sequence of
solutions to (\ref{eqnn1.1}) together with a criterion under which the limit
of stochastic integrals driven by convergent semimartingales is the stochastic
integral of the limits.

\section{Sigma convergence concept}

We begin this section by recalling some notion in connection with the
well-known concept of algebras with mean value.

Let $A$ be an algebra with mean value, that is, a closed subalgebra of bounded
uniformly continuous real-valued functions on $\mathbb{R}^{N}$, $\mathrm{BUC}%
\left(  \mathbb{R}^{N}\right)  $, which contains the constants, is translation
invariant and is such that any of its elements possesses a mean value in the
following sense: for every $u\in A$, the sequence $\left(  u^{\varepsilon
}\right)  _{\varepsilon>0}$ ($u^{\varepsilon}\left(  x\right)  =u\left(
\frac{x}{\varepsilon}\right)  $) weakly$\ast$ converges in $L^{\infty
}(\mathbb{R}^{N})$ to some real number $M(u)$ (called the mean value of $u$)
as $\varepsilon\rightarrow0$. The real number $M(u)$ is defined by
\[
M\left(  u\right)  =\lim_{R\rightarrow\infty}\frac{1}{\left\vert
B_{R}\right\vert }\int_{B_{R}}u\left(  y\right)  dy\text{ for }u\in A\text{,}%
\]
where $B_{R}$ is the open ball in $\mathbb{R}^{N}$ of radius $R$ centered at
the origin.

We consider an algebra with mean value $A$ on $\mathbb{R}^{N}$ and $m\in%
\mathbb{N}
$. We define
\[
A^{\infty}=\left\{  u\in\mathcal{C}^{\infty}\left(  \mathbb{R}^{N}\right)
:D_{y}^{\alpha}u\in A\text{ }\forall\alpha\in\mathbb{N}^{N}\right\}  \text{,}%
\]
a Fr\'{e}chet space when endowed with the locally convex topology defined by
the family of norms $\left\Vert \cdot\right\Vert _{m}$.

We consider $A_{y}$ (resp. $A_{\tau})$ an algebra with mean value on
$\mathbb{R}_{y}^{N}$ (resp. $\mathbb{R}_{\tau})$, $F$ a Banach space and
$\mathrm{BUC}(\mathbb{R}^{N};F)$ the Banach space of bounded uniformly
continuous functions $u:\mathbb{R}^{N}\rightarrow F$. We will use in this work
some tools as:

\begin{itemize}
\item[(1)] The product algebra with mean value $A_{y}\odot A_{\tau}$ which is
the closure in $\mathrm{BUC}(\mathbb{R}^{N+1})$ of the tensor product
$A_{y}\otimes A_{\tau}=\{\sum_{\text{finite}}u_{i}\otimes v_{i\text{ }}%
;u_{i}\in A_{y}$ and $v_{i\text{ }}\in A_{\tau}\}$, which is an algebra with
mean value on $\mathbb{R}^{N+1}.$

\item[(2)] The vector-valued algebra with mean value $A\left(  \mathbb{R}%
^{N};F\right)  $, which is the closure of $A\otimes F$ in $\mathrm{BUC}%
(\mathbb{R}^{N};F)$, where $A\otimes F$ is the space of functions of the form
$A\otimes F=\sum_{\text{finite}}u_{i}\otimes e_{i\text{ }},u_{i}\in A,e_{i}\in
F,$ with $(u_{i}\otimes e)\left(  y\right)  =u_{i}\left(  y\right)  e_{i}$ for
$y\in\mathbb{R}^{N}$. It is an easy exercise to show that $A_{y}\odot A_{\tau
}=A_{y}(\mathbb{R}^{N},A_{\tau})=A_{\tau}(\mathbb{R},A_{y})$.
\end{itemize}

The Besicovitch space $B_{A}^{p}(\mathbb{R}^{N};F)$ is the completion of
$A(\mathbb{R}^{N};F)$ with respect to the seminorm
\[
\left\Vert u\right\Vert _{p,F}=\left(  \lim_{R\rightarrow\infty}\frac
{1}{\left\vert B_{R}\right\vert }\int_{B_{R}}\left\Vert u\left(  y\right)
\right\Vert _{F}^{p}dy\right)  ^{\frac{1}{p}}\equiv\left(  M\left(  \left\Vert
u\right\Vert _{F}^{p}\right)  \right)  ^{\frac{1}{p}}\text{ for }u\in
A(\mathbb{R}^{N};F),\text{ }1\leq p<\infty\text{,}%
\]
and is a complete seminormed subspace of $L_{loc}^{p}(\mathbb{R}^{N};F)$.

We can therefore define the space $\mathcal{B}_{A}^{p}(\mathbb{R}^{N};F)$ as
follows:
\[
\mathcal{B}_{A}^{p}(\mathbb{R}^{N};F)=B_{A}^{p}(\mathbb{R}^{N};F)/\mathcal{N}%
\text{ (where }\mathcal{N}=\{u\in B_{A}^{p}(\mathbb{R}^{N};F):\left\Vert
u\right\Vert _{p,F}=0\}\text{),}%
\]
which is a Banach space endow with the norm $\left\Vert u+\mathcal{N}%
\right\Vert _{p,F}=\left\Vert u\right\Vert _{p,F}$ ($u\in B_{A}^{p}%
(\mathbb{R}^{N};F)$).

It is to be noted that the mean value $M:A\otimes F\rightarrow F$ extends by
continuity to a continuous linear mapping (still denoted by $M$) on $B_{A}%
^{p}(\mathbb{R}^{N};F)$ satisfying:
\[
T\left(  M\left(  u\right)  \right)  =M\left(  T\left(  u\right)  \right)
\text{ for all }T\in F^{\prime}\text{ and }u\in B_{A}^{p}(\mathbb{R}^{N};F).
\]

For a Hilbert space $H$, the space $\mathcal{B}_{A}^{2}(\mathbb{R}^{N};H)$ is
also a Hilbert space with inner product%
\[
\left(  u,v\right)  _{2}=M\left[  \left(  u,v\right)  _{H}\right]  \text{ for
any }u,v\in\mathcal{B}_{A}^{2}(\mathbb{R}^{N};H)\text{,}%
\]
where $\left(  ,\right)  _{H}$ is the inner product in $H$.

The Sobolev-Besicovitch space $\mathcal{B}_{A}^{1,p}(\mathbb{R}^{N})$ and its
subspace $\mathcal{B}_{A}^{1,p}(\mathbb{R}^{N})/\mathbb{R}$ are defined as follows:

\begin{itemize}
\item $\mathcal{B}_{A}^{1,p}(\mathbb{R}^{N})=\mathcal{B}_{A}^{1,p}%
(\mathbb{R}^{N};\mathbb{R})=\{u\in\mathcal{B}_{A}^{p}(\mathbb{R}^{N}%
):\frac{\overline{\partial}u}{\partial y_{i}}\in\mathcal{B}_{A}^{p}%
(\mathbb{R}^{N}),$ $1\leq i\leq N\}$, where $\frac{\overline{\partial}%
u}{\partial y_{i}}=\frac{\partial v}{\partial y_{i}}+\mathcal{N}$ for
$u=v+\mathcal{N}$ with $v\in B_{A}^{p}(\mathbb{R}^{N})$. Under the norm
\[
\left\Vert u\right\Vert _{\mathcal{B}_{A}^{1,p}}=\left(  \left\Vert
u\right\Vert _{p}^{p}+\sum_{i=1}^{N}\left\Vert \frac{\overline{\partial}%
u}{\partial y_{i}}\right\Vert _{p}^{p}\right)  ^{\frac{1}{p}}\ \ (u\in
\mathcal{B}_{A}^{1,p}(\mathbb{R}^{N})),
\]
$\mathcal{B}_{A}^{1,p}(\mathbb{R}^{N})$ is a Banach space which admits
$\mathcal{D}_{A}(\mathbb{R}^{N})=\varrho(A^{\infty})$ as a dense subspace,
$\varrho$ being the canonical mapping of $B_{A}^{p}(\mathbb{R}^{N})$ into
$\mathcal{B}_{A}^{p}(\mathbb{R}^{N})$ defined by $\varrho\left(  u\right)
=u+\mathcal{N}$.

\item $\mathcal{B}_{A}^{1,p}(\mathbb{R}^{N})/\mathbb{R=}\{u\in\mathcal{B}%
_{A}^{1,p}(\mathbb{R}^{N})$ with $M(u)=0\}$. Equipped with the norm
\[
\left\Vert u\right\Vert _{\#,p}=\left\Vert \overline{\nabla}_{y}u\right\Vert
_{p}:=\left[  \sum_{i=1}^{N}\left\Vert \frac{\overline{\partial}u}{\partial
y_{i}}\right\Vert _{p}^{p}\right]  ^{1/p}\;\;(u\in\mathcal{B}_{A}%
^{1,p}(\mathbb{R}^{N})/\mathbb{R}),
\]
$\mathcal{B}_{A}^{1,p}(\mathbb{R}^{N})/\mathbb{R}$ is a normed space which is
in general not complete. We denote by $\mathcal{B}_{\#A}^{1,p}(\mathbb{R}%
^{N})$ the completion of $\mathcal{B}_{A}^{1,p}(\mathbb{R}^{N})/\mathbb{R}$
with respect to $\left\Vert \cdot\right\Vert _{\#,p}$.
\end{itemize}

With the Sobolev space defined, we proceed forward and recall some facts about
ergodic algebras. An algebra with mean value $A$ on $\mathbb{R}^{N}$ is
ergodic if any invariant function $f\in\mathcal{B}_{A}^{1}\left(
\mathbb{R}^{N}\right)  $ is constant in $\mathcal{B}_{A}^{1}\left(
\mathbb{R}^{N}\right)  $. From now on, we assume that all the algebras with
mean value considered here are ergodic. For $u\in B_{A}^{p}(\mathbb{R}^{N}),$
$v\in B_{A}^{p}(\mathbb{R}^{N})^{N}$, we define the gradient operator
$\overline{\nabla}_{y}u$ and the divergence operator $\overline
{\operatorname{div}}_{y}v$ by%

\[
\overline{\nabla}_{y}u:=\left(  \frac{\overline{\partial}u}{\partial y_{1}%
},...,\frac{\overline{\partial}u}{\partial y_{N}}\right)  \text{, }%
\overline{\operatorname{div}}_{y}v:=\sum_{i=1}^{N}\frac{\overline{\partial
}v_{i}}{\partial y_{i}}%
\]
The following properties are satisfied.

\begin{itemize}
\item[(1)] The divergence operator sends continuously and linearly
$B_{A}^{p^{\prime}}(\mathbb{R}^{N})^{N}$ into $B_{A}^{p}(\mathbb{R}%
^{N})^{\prime}$.

\item[(2)] We have
\begin{equation}
\left\langle \overline{\operatorname{div}}_{y}v,u\right\rangle =-\left\langle
v,\overline{\nabla}_{y}u\right\rangle \text{, for any }u\in\mathcal{B}%
_{A}^{1,p}\left(  \mathbb{R}^{N}\right)  \text{ and }v=\left(  v_{i}\right)
\in\mathcal{B}_{A}^{p^{\prime}}\left(  \mathbb{R}^{N}\right)  ^{N}\text{,}
\label{eqn5}%
\end{equation}
where
\[
\left\langle v,\overline{\nabla}_{y}u\right\rangle :=\sum_{i=1}^{N}M\left(
v_{i}\frac{\overline{\partial}u}{\partial y_{i}}\right)  \text{.}%
\]

\end{itemize}

We are now able to recall the definition and main properties of $\Sigma
$-converge method. To this end, let $(\Omega,\mathcal{F},\mathbb{P})$ be a
probability space with expectation $\mathbb{E}$. We denote by $F\left(
\Omega\right)  $ the Banach space of bounded functions $f:\Omega
\rightarrow\mathbb{R}$, and by $B\left(  \Omega\right)  $ the closure in
$F\left(  \Omega\right)  $ of the vector space $H\left(  \Omega\right)  $
consisting of all finite linear combinations of characteristic functions
$1_{X}$ of sets $X\in\mathcal{F}$. $B(\Omega)$ is the Banach space of bounded
$\mathcal{F}$-measurable functions. We also define the space $B\left(
\Omega,Z\right)  $ of bounded $\left(  \mathcal{F},\mathcal{B}(Z)\right)
$-measurable functions $f:\Omega\rightarrow Z$, where $Z$ is a Banach space
equipped with the Borelians $\mathcal{B}(Z)$. For a Banach space $F$, the
space of $F$-valued random variables $u$ satisfying $\left\Vert u\right\Vert
_{F}\in L^{p}(\Omega,\mathcal{F},\mathbb{P})$ will be denoted by $L^{p}%
(\Omega,\mathcal{F},\mathbb{P};F)$ (or merely $L^{p}(\Omega;F)$, if there is
no danger of confusion). In all what follows, random variables will be
considered on the probability space $\left(  \Omega,\mathcal{F},\mathbb{P}%
\right)  $. The letter $E$ will throughout denote any ordinary sequence
$(\varepsilon_{n})_{n\in\mathbb{N}}$ with $0<\varepsilon_{n}\leq1$ and
$\varepsilon_{n}\rightarrow0$ as $n\rightarrow\infty$.

For a function $u\in L^{p}\left(  \Omega;L^{p}\left(  Q_{T};\mathcal{B}%
_{A}^{p}\left(  \mathbb{R}^{N+1}\right)  \right)  \right)  $ we denote by
$u\left(  x,t,\cdot,\omega\right)  $ (for fixed $\left(  x,t,\omega\right)
\in Q_{T}\times\Omega$) the element of $\mathcal{B}_{A}^{p}(\mathbb{R}^{N+1})$
defined by%
\[
u\left(  x,t,\cdot,\omega\right)  \left(  y,\tau\right)  =u\left(
x,t,y,\tau,\omega\right)  \text{ for }\left(  y,\tau\right)  \in
\mathbb{R}^{N+1}\text{.}%
\]

\begin{definition}
\label{d2.1}\emph{A\ sequence of random variables }$\left(  u_{\varepsilon
}\right)  _{\varepsilon>0}\subset L^{p}(\Omega;L^{p}(Q_{T}))$ $\left(  1\leq
p<\infty\right)  $\emph{\ is said to:}

\begin{itemize}
\item[(i)] weakly $\Sigma$-converge\emph{\ in }$L^{p}(Q_{T}\times\Omega
)$\emph{\ to some random variable }$u_{0}\in L^{p}(\Omega;L^{p}(Q_{T}%
;\mathcal{B}_{A}^{p}(\mathbb{R}^{N+1})))$\emph{\ if as }$\varepsilon
\rightarrow0,$\emph{\ we have }%
\begin{align}
&  \iint_{Q_{T}\times\Omega}u_{\varepsilon}\left(  x,t,\omega\right)  f\left(
x,t,\frac{x}{\varepsilon},\frac{t}{\varepsilon},\omega\right)  dxdtd\mathbb{P}%
\nonumber\\
&  \mathbb{\rightarrow}\iint_{Q_{T}\times\Omega}M(u_{0}(x,t,\cdot
,\omega)f(x,t,\cdot,\omega))dxdtd\mathbb{P} \label{eqn6}%
\end{align}
\emph{for every }$f\in L^{p^{\prime}}(\Omega;L^{p^{\prime}}(Q_{T};A))$\emph{,
}$\frac{1}{p}+\frac{1}{p^{\prime}}=1$\emph{. We express this by writing
}$u_{\varepsilon}\rightarrow u_{0}$\emph{\ in }$L^{p}(Q_{T}\times\Omega
)$\emph{-weak }$\Sigma$\emph{;}

\item[(ii)] strongly $\Sigma$-converge\emph{\ in }$L^{p}(Q_{T}\times\Omega
)$\emph{\ to }$u_{0}\in L^{p}(\Omega;L^{p}(Q_{T};\mathcal{B}_{A}%
^{p}(\mathbb{R}^{N+1})))$\emph{\ if (\ref{eqn6}) holds and further }%
\begin{equation}
\left\Vert u_{\varepsilon}\right\Vert _{L^{p}\left(  Q_{T}\times\Omega\right)
}\rightarrow\left\Vert u_{0}\right\Vert _{L^{p}\left(  Q_{T}\times
\Omega;\mathcal{B}_{A}^{p}\left(  \mathbb{R}_{y,\tau}^{N+1}\right)  \right)
}. \label{eqn9}%
\end{equation}
\emph{We express this by writing }$u_{\varepsilon}\rightarrow u_{0}$\emph{\ in
}$L^{p}(Q_{T}\times\Omega)$\emph{-strong }$\Sigma$\emph{.}
\end{itemize}
\end{definition}

\begin{remark}
\label{r2.2}\emph{The convergence (\ref{eqn6}) still holds true for }$f\in
B(\Omega;\mathcal{C}(\overline{Q}_{T};B_{A}^{p^{\prime},\infty}\left(
\mathbb{R}_{y,\tau}^{N+1}\right)  ))$\emph{, where }$B_{A}^{p^{\prime},\infty
}\left(  \mathbb{R}_{y,\tau}^{N+1}\right)  =B_{A}^{p^{\prime}}\left(
\mathbb{R}_{y,\tau}^{N+1}\right)  \cap L^{\infty}\left(  \mathbb{R}_{y,\tau
}^{N+1}\right)  $\emph{, }$\frac{1}{p}+\frac{1}{p^{\prime}}=1$\emph{.}
\end{remark}

\begin{remark}
\label{r2.3}\emph{If in Definition \ref{d2.1}, we choose test functions
independent of }$\omega$\emph{ and we do not integrate over }$\Omega$\emph{,
we get the usual }$\Sigma$\emph{-convergence method, that is, }$u_{\varepsilon
}\rightarrow u_{0}$\emph{\ in }$L^{p}(Q_{T})$\emph{-weak }$\Sigma$\emph{ if }%
\begin{equation}
\int_{Q_{T}}u_{\varepsilon}\left(  x,t,\omega\right)  f\left(  x,t,\frac
{x}{\varepsilon},\frac{t}{\varepsilon}\right)  dxdt\mathbb{\rightarrow}%
\int_{Q_{T}}M(u_{0}(x,t,\cdot,\omega)f(x,t,\cdot))dxdt\text{ }\mathbb{P}%
\text{\emph{-a.s.}}\nonumber
\end{equation}
\emph{for every }$f\in L^{p^{\prime}}(Q_{T};A)$\emph{. We define the strong
}$\Sigma$\emph{-convergence method accordingly.}
\end{remark}

The main properties of the sigma-convergence for stochastic processes can be
found in \cite{13} (see also \cite{JMAA, SAA}). They read as follows.

\begin{itemize}
\item[(SC)$_{1}$] For $1<p<\infty$, any sequence of random variables which is
bounded in $L^{p}\left(  \Omega;L^{p}\left(  Q_{T}\right)  \right)  $
possesses a weakly $\Sigma$-convergent subsequence.

\item[(SC)$_{2}$] Let $1<p<\infty$. Let $\left(  u_{\varepsilon}\right)
_{\varepsilon\in E}\subset L^{p}(\Omega;L^{p}(0,T;W_{0}^{1,p}(Q)))$ be a
sequence of random variables which satisfies the following estimate%
\[
\sup_{\varepsilon\in E}\mathbb{E}\left\Vert u_{\varepsilon}\right\Vert
_{L^{p}\left(  0,T;W_{0}^{1,p}\left(  Q\right)  \right)  }^{p}<\infty.
\]
Then there exist a subsequence $E^{\prime}$ from $E$ and a couple of random
variables $\left(  u_{0},u_{1}\right)  $ with $u_{0}\in L^{p}(\Omega
;L^{p}(0,T;W_{0}^{1,p}(Q)))$ and $u_{1}\in L^{p}(\Omega;L^{p}(Q_{T}%
;\mathcal{B}_{A_{\tau}}^{p}(\mathbb{R}_{\tau};\mathcal{B}_{\#A_{y}}%
^{1,p}(\mathbb{R}_{y}^{N}))))$ such that as $E^{\prime}\ni\varepsilon
\rightarrow0$,
\[
u_{\varepsilon}\rightarrow u_{0}\text{ in }L^{p}(\Omega;L^{p}(0,T;W_{0}%
^{1,p}(Q)))\text{-weak}%
\]
and%
\begin{equation}
\frac{\partial u_{\varepsilon}}{\partial x_{i}}\rightarrow\frac{\partial
u_{0}}{\partial x_{i}}+\frac{\overline{\partial}u_{1}}{\partial y_{i}}\text{
in }L^{p}(Q_{T}\times\Omega)\text{-weak }\Sigma,\text{ }1\leq i\leq N.
\label{eqn7}%
\end{equation}

\item[(SC)$_{3}$] Assume that the hypotheses of (SC)$_{2}$ above are
satisfied. Finally suppose further that $p\geq2$ and that there exist a
subsequence $E^{\prime}$ from $E$ and a random variable $u_{0}\in L^{p}%
(\Omega;L^{p}(0,T;W_{0}^{1,p}(Q)))$ such that as $E^{\prime}\ni\varepsilon
\rightarrow0$
\begin{equation}
u_{\varepsilon}\rightarrow u_{0}\text{ in }L^{2}(Q_{T}\times\Omega
)\text{-strong.} \label{eqn8}%
\end{equation}
Then there exist a subsequence of $E^{\prime}$ (not relabeled) and a
$L^{p}(Q_{T};\mathcal{B}_{A_{\tau}}^{p}(\mathbb{R}_{\tau};\mathcal{B}%
_{\#A_{y}}^{1,p}(\mathbb{R}_{y}^{N})))$-valued stochastic process $u_{1}\in
L^{p}(\Omega,L^{p}(Q_{T};\mathcal{B}_{A_{\tau}}^{p}(\mathbb{R}_{\tau
};\mathcal{B}_{\#A_{y}}^{1,p}(\mathbb{R}_{y}^{N}))))$ such that (\ref{eqn7})
holds when $E^{\prime}\ni\varepsilon\rightarrow0$.

\item[(SC)$_{4}$] Let $1<p,q<\infty$ and $r\geq1$ be such that $\frac{1}%
{r}=\frac{1}{p}+\frac{1}{q}\leq1$. Assume that $\left(  u_{\varepsilon
}\right)  _{\varepsilon>0}\subset$ $L^{q}(Q_{T}\times\Omega)$ is weakly
$\Sigma$-convergent in $L^{q}(Q_{T}\times\Omega)$ to some $u_{0}\in
L^{q}(Q_{T}\times\Omega;\mathcal{B}_{A}^{q}(\mathbb{R}_{y,\tau}^{N+1}))$ and
$\left(  v_{\varepsilon}\right)  _{\varepsilon>0}\subset L^{p}(Q_{T}%
\times\Omega)$ is strongly $\Sigma$-convergent in $L^{p}(Q_{T}\times\Omega)$
to some $v_{0}\in L^{p}(Q_{T}\times\Omega;\mathcal{B}_{A}^{p}(\mathbb{R}%
_{y,\tau}^{N+1}))$. Then the sequence $\left(  u_{\varepsilon}v_{\varepsilon
}\right)  _{\varepsilon>0}$ is weakly $\Sigma$-convergent in $L^{r}%
(Q_{T}\times\Omega)$ to $u_{0}v_{0}$.
\end{itemize}

The above properties are the random counterpart of the same properties already
derived in \cite{Hom1} (see also \cite{EJDE04}) in the deterministic
setting.\medskip

This being so, it is a well-known fact that if $A=A_{y}\odot A_{\tau}$ is an
algebra with mean value on $\mathbb{R}^{N+1}$, then the natural choice of test
functions in the homogenization process is given by $\Phi_{\varepsilon
}(x,t,\omega)=\psi_{0}(x,t,\omega)+\varepsilon\psi_{1}(x,t,x/\varepsilon
,t/\varepsilon,\omega)$, $(x,t,\omega)\in Q_{T}\times\Omega$. However, due to
the following result, although the coefficients in our problem depend on the
fast time variable $\tau=t/\varepsilon$, the function $\psi_{1}$ will not
depend on $\tau=t/\varepsilon$. This is a consequence of the control of the
damping term.

\begin{lemma}
\label{l2.1}Let $\left(  u_{0},u_{1}\right)  $ be a couple of random variables
satisfying \emph{(\ref{eqn7})}, and $\left(  u_{\varepsilon}\right)
_{\varepsilon\in E}$ be a bounded sequence in $L^{2}\left(  \Omega
;L^{2}\left(  0,T;H_{0}^{1}(Q)\right)  \right)  $ such that
\[
\sup_{\varepsilon\in E}\mathbb{E}\left\Vert \nabla u_{\varepsilon}^{\prime
}\right\Vert _{L^{2}(Q_{T})}^{2}<\infty.
\]
Then we have $u_{1}\in L^{2}(\Omega;L^{2}(Q_{T};\mathcal{B}_{\#A_{y}}%
^{1,2}(\mathbb{R}_{y}^{N})))+L^{2}(\Omega;L^{2}(Q_{T};\mathcal{B}_{A_{\tau}%
}^{2}(\mathbb{R}_{\tau})/\mathbb{R}))$.
\end{lemma}

\begin{proof}
For $\varphi\in\mathcal{C}_{0}^{\infty}(Q_{T})$, $\psi\in A_{y}^{\infty}$,
$\phi\in A_{\tau}^{\infty}$, $\eta\in B\left(  \Omega\right)  $ and $1\leq
i\leq N$, one has
\begin{align}
&  \varepsilon\int_{Q_{T}\times\Omega}\frac{\partial^{2}u_{\varepsilon}%
}{\partial t\partial x_{i}}(x,t,\omega)\varphi(x,t)\psi(\frac{x}{\varepsilon
})\phi(\frac{t}{\varepsilon})\eta\left(  \omega\right)  dxdtd\mathbb{P}%
\label{eqnb}\\
&  +\varepsilon\int_{Q_{T}\times\Omega}\frac{\partial u_{\varepsilon}%
}{\partial x_{i}}(x,t,\omega)\frac{\partial\varphi}{\partial t}(x,t)\psi
(\frac{x}{\varepsilon})\phi(\frac{t}{\varepsilon})\eta\left(  \omega\right)
dxdtd\mathbb{P}\nonumber\\
&  =-\int_{Q_{T}\times\Omega}\frac{\partial u_{\varepsilon}}{\partial x_{i}%
}(x,t,\omega)\varphi(x,t)\psi(\frac{x}{\varepsilon})\frac{d\phi}{d\tau}%
(\frac{t}{\varepsilon})\eta\left(  \omega\right)  dxdtd\mathbb{P}.\nonumber
\end{align}
Since $(\nabla u_{\varepsilon})_{\varepsilon\in E}$ and $(\nabla
u_{\varepsilon}^{\prime})_{\varepsilon\in E}$ are bounded in $L^{2}%
(Q_{T}\times\Omega)$, passing to the limit in (\ref{eqnb}), as $E^{\prime}%
\ni\varepsilon\rightarrow0$ (the subsequence in (SC)$_{2}$ above), the left
hand-side of the above equality goes to zero, so using the arbitrariness of
$\varphi$ and $\psi$, we have
\[
M_{\tau}\left(  \left(  \frac{\partial u_{0}}{\partial x_{i}}+\frac
{\overline{\partial}u_{1}}{\partial y_{i}}\right)  \frac{d\phi}{d\tau}\right)
=0\text{ in }Q_{T}\times\Omega\times\mathbb{R}_{y}^{N}\text{ for all }\phi\in
A_{\tau}^{\infty}\text{ }(1\leq i\leq N).
\]
This implies that $\frac{\overline{\partial}u_{1}}{\partial y_{i}}$ does not
depend on $\tau$, hence $\frac{\overline{\partial}u_{1}}{\partial y_{i}}\in
L^{2}\left(  \Omega;L^{2}\left(  Q_{T};\mathcal{B}_{A_{y}}^{1,2}\left(
\mathbb{R}_{y}^{N}\right)  \right)  \right)  $. Letting $v_{1}=M_{\tau}\left(
u_{1}\right)  \in L^{2}(\Omega;L^{2}(Q_{T};\mathcal{B}_{\#A_{y}}%
^{1,2}(\mathbb{R}_{y}^{N})))$ we have $\overline{\nabla}_{y}v_{1}=M_{\tau
}\left(  \overline{\nabla}_{y}u_{1}\right)  =\overline{\nabla}_{y}u_{1}$, so
that $\overline{\nabla}_{y}\left(  u_{1}-v_{1}\right)  =0$. Since the algebra
$A_{y}$ is ergodic, we obtain $u_{1}-v_{1}=\chi\in L^{2}\left(  \Omega
;L^{2}\left(  Q_{T};\mathcal{B}_{A_{\tau}}^{2}(\mathbb{R})/\mathbb{R}\right)
\right)  $, thus, $u_{1}=v_{1}+\chi\in L^{2}(\Omega;L^{2}(Q_{T};\mathcal{B}%
_{\#A_{y}}^{1,2}(\mathbb{R}_{y}^{N})))+L^{2}\left(  \Omega;L^{2}\left(
Q_{T},\mathcal{B}_{A_{\tau}}^{2}(\mathbb{R}_{\tau})/\mathbb{R}\right)
\right)  $.
\end{proof}

\begin{remark}
\label{r2.4}\emph{Lemma \ref{l2.1} allows us to choose the function }$u_{1}%
$\emph{\ in }$L^{2}(\Omega;L^{2}(Q_{T};\mathcal{B}_{\#A_{y}}^{1,2}%
(\mathbb{R}_{y}^{N})))$\emph{. Indeed, as seen in the proof of that lemma, we
may replace }$u_{1}$\emph{\ by }$v_{1}=M_{\tau}(u_{1})$\emph{\ so that
}$\overline{\nabla}_{y}v_{1}=\overline{\nabla}_{y}u_{1}$\emph{. As a result,
we will choose the test functions (in the homogenization process) that are
independent of }$\tau$\emph{.}
\end{remark}

\section{Homogenization results}

\subsection{Passage to the limit}

The passage to the limit will be made under the assumption (\textbf{A3}) below
characterizing essentially the distribution of the microstructures in the
medium $Q$ and the behavior on the fast time scale. It reads as follows:

\begin{itemize}
\item[\textbf{(A3)}] $a_{ij}(x,\cdot)\in B_{A_{y}}^{2}(\mathbb{R}_{y}^{N})$
for every $x\in\overline{Q}$, where $A_{y}$ is an algebra with mean value on
$\mathbb{R}^{N}$, for any $v\in\mathbb{R}$ the functions $(y,\tau)\mapsto
f(y,\tau,v)$ and $(y,\tau)\mapsto g_{k}(y,\tau,v)$ belong to $B_{A}%
^{2}(\mathbb{R}_{y,\tau}^{N+1})$ where $A=A_{y}\odot A_{\tau}$ is a product
algebra with mean value defined as a closure in the sup norm in $\mathbb{R}%
^{N+1}$ of the tensor product $A_{y}\otimes A_{\tau}$ and $g=(g_{k})_{1\leq
k\leq m}$.
\end{itemize}

Starting with the sequence of probability measures $\left(  \pi^{\varepsilon
}\right)  $ defined by (\ref{eqqq}),\emph{\ }since it is tight, the
Prokhorov's theorem \cite{12} yields the existence of a subsequence $\left(
\pi^{\varepsilon_{n}}\right)  $ of $\left(  \pi^{\varepsilon}\right)  $ that
weakly converges to a probability measure $\pi$ in $S$. Moreover, by
Skorokhod's theorem \cite{1}, there exist a probability space $\left(
\overline{\Omega},\overline{\mathcal{F}},\overline{\mathbb{P}}\right)  $, and
random variables $\left(  W_{\varepsilon_{n}},u_{\varepsilon_{n}%
},u_{\varepsilon_{n}}^{\prime}\right)  $ and $\left(  \overline{W},u_{0}%
,u_{0}^{\prime}\right)  $ with values in $S$ such that:

\begin{itemize}
\item[(\textbf{C1})] $\pi^{\varepsilon_{n}}$ and $\pi$ are respectively the
probability law of $(W_{\varepsilon_{n}},u_{\varepsilon_{n}},u_{\varepsilon
_{n}}^{\prime})$ and $(\overline{W},u_{0},u_{0}^{\prime})$,

\item[(\textbf{C2})] $u_{\varepsilon_{n}}\rightarrow u_{0}$ in $L^{2}%
(0,T;L^{2}(Q))\cap\mathcal{C}(0,T;H^{-1}(Q))$ $\overline{\mathbb{P}}$-a.s.,

\item[(\textbf{C3})] $u_{\varepsilon_{n}}^{\prime}\rightarrow u_{0}^{\prime}$
in $\mathcal{C}(0,T;L^{2}(Q))$ $\overline{\mathbb{P}}$-a.s.,

\item[(\textbf{C4})] $W_{\varepsilon_{n}}\rightarrow\overline{W}$ in
$\mathcal{C}(0,T;\mathbb{R}^{m})$ $\overline{\mathbb{P}}$-a.s.
\end{itemize}

With this in mind, we deduce that $\{W_{\varepsilon_{n}}\}$ is a sequence
of\ $m$-dimensional Wiener process. Now, set
\[
\overline{\mathcal{F}}_{t}=\sigma\{\overline{W}(s),u_{0}(s),u_{0}^{\prime
}(s)\}_{s\in\lbrack0,t]}.
\]
Then arguing as in \cite{10}, we may prove that $\overline{W}\left(  t\right)
$ is a $\overline{\mathcal{F}}_{t}$-standard Wiener process. By the same way
of proceeding as in \cite{2'}, we can show that $u_{\varepsilon_{n}}$
(obtained above) satisfies
\begin{align}
&  \left(  u_{\varepsilon_{n}}^{\prime}(t),\phi\right)  +\int_{0}^{t}\left[
\left(  P^{\varepsilon_{n}}u_{\varepsilon_{n}}\left(  \tau\right)
,\phi\right)  -\left(  \nabla u_{\varepsilon_{n}}^{\prime}\left(  \tau\right)
,\nabla\phi\right)  \right]  d\tau\label{eqn32}\\
&  =\left(  u^{1},\phi\right)  +\int_{0}^{t}\left(  f\left(  \frac
{x}{\varepsilon_{n}},\frac{\tau}{\varepsilon_{n}},u_{\varepsilon_{n}}^{\prime
}\right)  ,\phi\right)  d\tau+\left(  \int_{0}^{t}g\left(  \frac
{x}{\varepsilon_{n}},\frac{\tau}{\varepsilon_{n}},u_{\varepsilon_{n}}^{\prime
}\right)  dW_{\varepsilon_{n}}\left(  \tau\right)  ,\phi\right) \nonumber
\end{align}
for any $\phi\in H_{0}^{1}(Q)$ and for almost all $(\overline{\omega}%
,t)\in\overline{\Omega}\times\lbrack0,T]$.

Since $u_{\varepsilon_{n}}$ satisfies (\ref{eqn32}), then we obtain from the
application of Prokhorov's and Skhorokhod's compactness results that
$u_{\varepsilon_{n}}$ and $u_{\varepsilon_{n}}^{\prime}$ satisfy the a priori
estimates
\begin{equation}
\overline{\mathbb{E}}\sup_{0\leq t\leq T}\left\Vert u_{\varepsilon_{n}}\left(
t\right)  \right\Vert _{H_{0}^{1}(Q)}^{4}\leq C\text{, }\overline{\mathbb{E}%
}\sup_{0\leq t\leq T}\left\Vert u_{\varepsilon_{n}}^{\prime}\left(  t\right)
\right\Vert _{L^{2}(Q)}^{4}\leq C\text{,} \label{eqn40}%
\end{equation}%
\begin{equation}
\overline{\mathbb{E}}\sup_{\left\vert \theta\right\vert \leq\delta}\int
_{0}^{T}\left\Vert u_{\varepsilon_{n}}^{\prime}(t+\theta)-u_{\varepsilon_{n}%
}^{\prime}(t)\right\Vert _{H^{-1}(Q)}^{2}dt\leq C\delta\text{.} \label{eqn40'}%
\end{equation}
We infer from these estimates that
\begin{align*}
u_{\varepsilon_{n}}  &  \rightarrow u_{0}\text{ in }L^{2}\left(
\overline{\Omega};L^{\infty}\left(  0,T;H_{0}^{1}(Q)\right)  \right)
\text{-weak}\ast\text{,}\\
u_{\varepsilon_{n}}  &  \rightarrow u_{0}\text{ in }L^{2}\left(
\overline{\Omega};L^{\infty}\left(  0,T;L^{2}(Q)\right)  \right)
\text{-weak}\ast\text{,}\\
u_{\varepsilon_{n}}^{\prime}  &  \rightarrow u_{0}^{\prime}=\frac{\partial
u_{0}}{\partial t}\text{ in }L^{2}\left(  \overline{\Omega};L^{\infty}\left(
0,T;L^{2}(Q)\right)  \right)  \text{-weak}\ast\text{.}%
\end{align*}
The following convergence results are therefore consequences of the relations
(\ref{eqn40}), (\ref{eqn40'}), Vitali's theorem
\begin{align}
u_{\varepsilon_{n}}  &  \rightarrow u_{0}\text{ in }L^{2}(\overline{\Omega
};L^{\infty}(0,T;L^{2}(Q)))\text{-strong,}\label{eqn41}\\
u_{\varepsilon_{n}}^{\prime}  &  \rightarrow u_{0}^{\prime}\text{ in }%
L^{2}(\overline{\Omega};L^{\infty}(0,T;L^{2}(Q)))\text{-strong.}\nonumber
\end{align}
Hence for almost all $(\omega,t)\in\overline{\Omega}\times\lbrack0,T]$, we
obtain
\begin{align}
u_{\varepsilon_{n}}  &  \rightarrow u_{0}\text{ \ in }L^{2}(Q)\text{-strong,}%
\label{eqn42}\\
u_{\varepsilon_{n}}^{\prime}  &  \rightarrow u_{0}^{\prime}\text{ \ in }%
L^{2}(Q)\text{-strong,}\nonumber
\end{align}
with respect to the measure $d\overline{\mathbb{P}}\otimes dt$.

\begin{lemma}
\label{l4.1}Let $(u_{\varepsilon})_{\varepsilon}$ be a sequence in
$L^{2}(Q_{T}\times\Omega)$ such that $u_{\varepsilon}^{\prime}\rightarrow
u_{0}^{\prime}$ in $L^{2}(Q_{T}\times\Omega)$ as $\varepsilon\rightarrow0$
where $u_{0}^{\prime}\in L^{2}(Q_{T}\times\Omega)$. Then, as $\varepsilon
\rightarrow0$, the following holds true:
\begin{equation}
f^{\varepsilon}(\cdot,\cdot,u_{\varepsilon}^{\prime})\rightarrow f(\cdot
,\cdot,u_{0}^{\prime})\text{ in }L^{2}(Q_{T})\text{-strong }\Sigma\text{
\ }\mathbb{P}\text{-a.s.} \label{eqn42'}%
\end{equation}
and, for each $1\leq k\leq m$,
\begin{equation}
g_{k}^{\varepsilon}(\cdot,\cdot,u_{\varepsilon}^{\prime})\rightarrow
g_{k}(\cdot,\cdot,u_{0}^{\prime})\text{ in }L^{2}(Q_{T}\times\Omega
)\text{-strong }\Sigma\text{.} \label{eqn42''}%
\end{equation}

\end{lemma}

\begin{proof}
Let us first prove (\ref{eqn42''}). It amounts in checking the following:

\begin{itemize}
\item[(i)] $g_{k}^{\varepsilon}(\cdot,\cdot,u_{\varepsilon}^{\prime
})\rightarrow g_{k}(\cdot,\cdot,u_{0}^{\prime})$ in $L^{2}(Q_{T}\times\Omega
)$-weak $\Sigma$

\item[(ii)] $\left\Vert g_{k}^{\varepsilon}(\cdot,\cdot,u_{\varepsilon
}^{\prime})\right\Vert _{L^{2}(Q_{T}\times\Omega)}\rightarrow\left\Vert
g_{k}(\cdot,\cdot,u_{\varepsilon}^{\prime})\right\Vert _{L^{2}(Q_{T}%
\times\Omega;\mathcal{B}_{A}^{2}(\mathbb{R}_{y,\tau}^{N+1}))}$.
\end{itemize}

We first consider (i). Let $u\in B(\Omega;\mathcal{C}\left(  \overline{Q}%
_{T}\right)  )$; then the function $(x,t,y,\tau,\omega)\mapsto g_{k}%
(y,\tau,u(x,t,\omega))$ lies in $B(\Omega;\mathcal{C}(\overline{Q}_{T}%
;B_{A}^{2,\infty}(\mathbb{R}_{y,\tau}^{N+1})))$, so that we have
$g_{k}^{\varepsilon}(\cdot,\cdot,u)\rightarrow g_{k}(\cdot,\cdot,u)$ in
$L^{2}(Q_{T}\times\Omega)$-weak $\Sigma$ as $\varepsilon\rightarrow0$. Using
the density of $B(\Omega;\mathcal{C}\left(  \overline{Q}_{T}\right)  )$ in
$L^{2}(Q_{T}\times\Omega)$ (which allows us to approximate $u_{0}^{\prime}$ in
$L^{2}(Q_{T}\times\Omega)$ by a strongly convergent sequence extracted from
$B(\Omega;\mathcal{C}\left(  \overline{Q}_{T}\right)  )$), we obtain the
following result:
\begin{equation}
g_{k}^{\varepsilon}(\cdot,\cdot,u_{0}^{\prime})\rightarrow g_{k}(\cdot
,\cdot,u_{0}^{\prime})\text{ in }L^{2}(Q_{T}\times\Omega)\text{-weak }%
\Sigma\text{ as }\varepsilon\rightarrow0\text{.} \label{eqnq}%
\end{equation}
Next, for $h\in L^{2}(\Omega;L^{2}(Q_{T};A))$, we have
\begin{align*}
&  \int_{Q_{T}\times\overline{\Omega}}g_{k}^{\varepsilon}(\cdot,\cdot
,u_{\varepsilon}^{\prime})h^{\varepsilon}dxdtd\mathbb{P}-\int_{Q_{T}%
\times\overline{\Omega}}M(g_{k}(\cdot,\cdot,u_{0}^{\prime})h)dxdtd\mathbb{P}\\
&  =\int_{Q_{T}\times\overline{\Omega}}(g_{k}^{\varepsilon}(\cdot
,\cdot,u_{\varepsilon}^{\prime})-g_{k}^{\varepsilon}(\cdot,\cdot,u_{0}%
^{\prime}))h^{\varepsilon}dxdtd\mathbb{P+}\int_{Q_{T}\times\overline{\Omega}%
}g_{k}^{\varepsilon}(\cdot,\cdot,u_{0}^{\prime})h^{\varepsilon}dxdtd\mathbb{P}%
\\
&  -\int_{Q_{T}\times\overline{\Omega}}M(g_{k}(\cdot,\cdot,u_{0}^{\prime
})h)dxdtd\mathbb{P}\\
&  =I_{1}+I_{2}%
\end{align*}
where
\begin{equation}
\left\vert I_{1}\right\vert =\left\vert \int_{Q_{T}\times\overline{\Omega}%
}(g_{k}^{\varepsilon}(\cdot,\cdot,u_{\varepsilon}^{\prime})-g_{k}%
^{\varepsilon}(\cdot,\cdot,u_{0}^{\prime}))h^{\varepsilon}dxdtd\mathbb{P}%
\right\vert \leq C\left\Vert u_{\varepsilon}^{\prime}-u_{0}^{\prime
}\right\Vert _{L^{2}(Q_{T}\times\Omega)^{N}}\left\Vert h^{\varepsilon
}\right\Vert _{L^{2}(Q_{T}\times\Omega)} \label{eqnqq}%
\end{equation}
and
\[
I_{2}=\int_{Q_{T}\times\overline{\Omega}}g_{k}^{\varepsilon}(\cdot,\cdot
,u_{0}^{\prime})h^{\varepsilon}dxdtd\mathbb{P}-\int_{Q_{T}\times
\overline{\Omega}}M(g_{k}(\cdot,\cdot,u_{0}^{\prime})h)dxdtd\mathbb{P}%
\]
Since $u_{\varepsilon}^{\prime}\rightarrow u_{0}^{\prime}$ in $L^{2}%
(Q_{T}\times\Omega)$-strong as $\varepsilon\rightarrow0$, and using
(\ref{eqnq}) and (\ref{eqnqq}) we obtain (i). Now, for (ii), we have to show
that
\begin{equation}
\int_{Q_{T}\times\Omega}\left\vert g_{k}^{\varepsilon}(\cdot,\cdot
,u_{\varepsilon}^{\prime})\right\vert ^{2}dxdtd\mathbb{P}\rightarrow
\int_{Q_{T}\times\Omega}M(\left\vert g_{k}(\cdot,\cdot,u_{0}^{\prime
})\right\vert ^{2})dxdtd\mathbb{P}. \label{01}%
\end{equation}
As above, if we choose $u\in B(\Omega;\mathcal{C}\left(  \overline{Q}%
_{T}\right)  )$, then the function $g_{k}(\cdot,\cdot,u):(x,t,y,\tau
,\omega)\mapsto g_{k}(y,\tau,u(x,t,\omega))$ is an element of $B(\Omega
;\mathcal{C}(\overline{Q}_{T};B_{A}^{2,\infty}(\mathbb{R}_{y,\tau}^{N+1})))$,
so that
\[
g_{k}^{\varepsilon}(\cdot,\cdot,u)\rightarrow g_{k}(\cdot,\cdot,u)\text{ in
}L^{2}(Q_{T}\times\Omega)\text{-strong }\Sigma.
\]
Now, we choose a sequence $(v_{n})_{n}\subset B(\Omega;\mathcal{C}\left(
\overline{Q}_{T}\right)  )$ satisfying
\begin{equation}
v_{n}\rightarrow u_{0}^{\prime}\text{ in }L^{2}(Q_{T}\times\Omega)\text{ as
}n\rightarrow\infty\text{.} \label{0.2}%
\end{equation}
We have
\begin{align*}
\int_{Q_{T}\times\Omega}\left\vert g_{k}^{\varepsilon}(\cdot,\cdot
,u_{\varepsilon}^{\prime})\right\vert ^{2}dxdtd\mathbb{P}  &  =\int
_{Q_{T}\times\Omega}g_{k}^{\varepsilon}(\cdot,\cdot,u_{\varepsilon}^{\prime
})(g_{k}^{\varepsilon}(\cdot,\cdot,u_{\varepsilon}^{\prime})-g_{k}%
^{\varepsilon}(\cdot,\cdot,u_{0}^{\prime}))dxdtd\mathbb{P}\\
&  +\int_{Q_{T}\times\Omega}g_{k}^{\varepsilon}(\cdot,\cdot,u_{\varepsilon
}^{\prime})(g_{k}^{\varepsilon}(\cdot,\cdot,u_{0}^{\prime})-g_{k}%
^{\varepsilon}(\cdot,\cdot,v_{n}))dxdtd\mathbb{P}\\
&  +\int_{Q_{T}\times\Omega}g_{k}^{\varepsilon}(\cdot,\cdot,u_{\varepsilon
}^{\prime})g_{k}^{\varepsilon}(\cdot,\cdot,v_{n})dxdtd\mathbb{P}\\
&  =J_{1}+J_{2}+J_{3}.
\end{align*}
As above, using the boundedness of the sequence $(g_{k}^{\varepsilon}%
(\cdot,\cdot,u_{\varepsilon}^{\prime}))_{\varepsilon}$ in $L^{2}(Q_{T}%
\times\Omega)$ and the Lipschitz property of $g_{k}$ associated to the strong
convergence of $u_{\varepsilon}^{\prime}$ towards $u_{0}^{\prime}$ in
$L^{2}(Q_{T}\times\Omega)$, we get that $J_{1}\rightarrow0$ as $\varepsilon
\rightarrow0$. The same argument as above associated to the strong convergence
(\ref{0.2}) yield $J_{2}\rightarrow0$. As for $J_{3}$, we use the weak
$\Sigma$-convergence (i) with test function $g_{k}(\cdot,\cdot,v_{n})$ to get
\[
J_{3}\rightarrow\int_{Q_{T}\times\Omega}M(g_{k}(\cdot,\cdot,u_{0}^{\prime
})g_{k}(\cdot,\cdot,v_{n}))dxdtd\mathbb{P}\text{ when }\varepsilon
\rightarrow0.
\]
In the above convergence result, we let finally $n\rightarrow\infty$ to get
\[
J_{3}\rightarrow\int_{Q_{T}\times\Omega}M(\left\vert g_{k}(\cdot,\cdot
,u_{0}^{\prime})\right\vert ^{2})dxdtd\mathbb{P}\mathbf{.}%
\]
This shows (ii), and hence (\ref{eqn42''}). Noticing that up to a subsequence,
we have $u_{\varepsilon}^{\prime}\rightarrow u_{0}^{\prime}$ in $L^{2}(Q_{T}%
)$-strong $\mathbb{P}$-a.s., we may hence proceed as above to obtain
(\ref{eqn42'}), thereby completing the proof.
\end{proof}

The next result deals with the convergence of the stochastic term.

\begin{lemma}
\label{l3.2}Let $(\Omega,\mathcal{F},\mathbb{P})$ be a probability space. Let
$E$ be a fundamental sequence, and let $(u_{\varepsilon})_{\varepsilon\in E}$
be a sequence in $L^{2}(Q_{T}\times\Omega)$ satisfying
\begin{equation}
u_{\varepsilon}^{\prime}\rightarrow u_{0}^{\prime}\text{ in }L^{2}(Q_{T}%
\times\Omega)\text{-strong as }E\ni\varepsilon\rightarrow0\text{.}\label{0.3}%
\end{equation}
Let $(W^{\varepsilon})_{\varepsilon\in E}$ be a sequence of $m$-dimensional
Wiener process on the probability space $(\Omega,\mathcal{F},\mathbb{P})$
satisfying
\begin{equation}
W^{\varepsilon}\rightarrow W\text{ in }\mathcal{C}(0,T;\mathbb{R}^{m})\text{
as }E\ni\varepsilon\rightarrow0\text{ }\mathbb{P}\text{-a.s.}\label{0.4}%
\end{equation}
where $W$ is a $m$-dimensional Wiener process on the probability space
$(\Omega,\mathcal{F},\mathbb{P})$. Then, for any $\psi_{0}\in\mathcal{C}%
_{0}^{\infty}(Q_{T})$, we have, as $E\ni\varepsilon\rightarrow0$,
\begin{equation}
\int_{Q_{T}}g^{\varepsilon}(\cdot,\cdot,u_{\varepsilon}^{\prime})\psi
_{0}dxdW^{\varepsilon}\rightarrow\int_{Q_{T}}\widetilde{g}(u_{0}^{\prime}%
)\psi_{0}dxdW\text{ in law,}\label{0.5}%
\end{equation}
where $\widetilde{g}(u_{0}^{\prime})=(\widetilde{g}_{k}(u_{0}^{\prime
}))_{1\leq k\leq m}$ is defined by $\widetilde{g}_{k}(u_{0}^{\prime
})=(M(\left\vert g_{k}(\cdot,\cdot,u_{0}^{\prime})\right\vert ^{2})^{1/2}$.
\end{lemma}

\begin{proof}
We proceed in three steps.

\textit{Step 1}. Let us first show that
\[
\int_{Q_{T}}g^{\varepsilon}(\cdot,\cdot,u_{\varepsilon}^{\prime})\psi
_{0}dxdW\rightarrow\int_{Q_{T}}\widetilde{g}(u_{0}^{\prime})\psi_{0}dxdW\text{
in law.}%
\]
We have
\begin{align*}
\int_{Q_{T}}g^{\varepsilon}(\cdot,\cdot,u_{\varepsilon}^{\prime})\psi_{0}dxdW
&  =\int_{Q_{T}}(g^{\varepsilon}(\cdot,\cdot,u_{\varepsilon}^{\prime
})-g^{\varepsilon}(\cdot,\cdot,u_{0}^{\prime}))\psi_{0}dxdW\\
&  +\int_{Q_{T}}g^{\varepsilon}(\cdot,\cdot,u_{0}^{\prime})\psi_{0}dxdW\\
&  =I_{1}^{\varepsilon}+I_{2}^{\varepsilon}.
\end{align*}
Concerning $I_{1}^{\varepsilon}$ we have
\begin{align*}
\mathbb{E}\left\vert I_{1}^{\varepsilon}\right\vert ^{2}  &  =\int_{Q_{T}%
}\mathbb{E}(\left\vert g^{\varepsilon}(\cdot,\cdot,u_{\varepsilon}^{\prime
})-g^{\varepsilon}(\cdot,\cdot,u_{0}^{\prime})\right\vert ^{2})\psi_{0}%
^{2}dxdt\\
&  \leq C\int_{Q_{T}}\mathbb{E}(\left\vert u_{\varepsilon}^{\prime}%
-u_{0}^{\prime}\right\vert ^{2})\psi_{0}^{2}dxdt\\
&  \leq C\left\Vert \psi_{0}\right\Vert _{\infty}^{2}\int_{Q_{T}}%
\mathbb{E}(\left\vert u_{\varepsilon}^{\prime}-u_{0}^{\prime}\right\vert
^{2})dxdt\rightarrow0\text{ as }E\ni\varepsilon\rightarrow0,
\end{align*}
where in the last inequality above, we have used the convergence result
(\ref{0.3}). It follows that, up to a subsequence of $E$ not relabeled,
$I_{1}^{\varepsilon}\rightarrow0$ $\mathbb{P}$-a.s. when $E\ni\varepsilon
\rightarrow0$.

Now, as for $I_{2}^{\varepsilon}$, we treat each term of the sum separately,
that is, we consider each
\[
I_{2,k}^{\varepsilon}=\int_{Q_{T}}g_{k}^{\varepsilon}(\cdot,\cdot
,u_{0}^{\prime})\psi_{0}dxdW_{k},\ 1\leq k\leq m.
\]
Each of these integrals is a Gaussian $\mathcal{N}(0,\sigma_{\varepsilon}%
^{2})$ where
\[
\sigma_{\varepsilon}^{2}=\mathbb{E}\left\vert I_{2,k}^{\varepsilon}\right\vert
^{2}=\int_{Q_{T}}\mathbb{E}(\left\vert g_{k}^{\varepsilon}(\cdot,\cdot
,u_{0}^{\prime})\right\vert ^{2})\psi_{0}^{2}dxdt.
\]
Using the convergence result (\ref{eqn42''}) in Lemma \ref{l4.1}, we get that
\[
\sigma_{\varepsilon}^{2}\rightarrow\int_{Q_{T}}\mathbb{E}(\left\vert
\widetilde{g}_{k}(u_{0}^{\prime})\right\vert ^{2})\psi_{0}^{2}dxdt\text{ as
}E\ni\varepsilon\rightarrow0.
\]
It follows from the martingale representation theorem that
\begin{equation}
\int_{Q_{T}}g_{k}^{\varepsilon}(\cdot,\cdot,u_{0}^{\prime})\psi_{0}%
dxdW_{k}\rightarrow\int_{Q_{T}}\widetilde{g}_{k}(u_{0}^{\prime})\psi
_{0}dxdW_{k}\text{ in law.}\label{0.6}%
\end{equation}
We recall that a sequence of Gaussian $\mathcal{N}(m_{\varepsilon}%
,\sigma_{\varepsilon}^{2})$ converges in law to the Gaussian $\mathcal{N}%
(m,\sigma^{2})$ if and only if $m_{\varepsilon}\rightarrow m$ and
$\sigma_{\varepsilon}^{2}\rightarrow\sigma^{2}$. This can be verified by using
the characteristic function $\Phi_{m_{\varepsilon},\sigma_{\varepsilon}^{2}%
}(t)=\exp(im_{\varepsilon}t-\sigma_{\varepsilon}^{2}t^{2}/2)$ of
$\mathcal{N}(m_{\varepsilon},\sigma_{\varepsilon}^{2})$. We therefore infer
from (\ref{0.6}) that, as $E\ni\varepsilon\rightarrow0$,
\[
I_{2}^{\varepsilon}\rightarrow\int_{Q_{T}}\widetilde{g}(u_{0}^{\prime}%
)\psi_{0}dxdW\text{ in law,}%
\]
where $\widetilde{g}(u_{0}^{\prime})=(\widetilde{g}_{k}(u_{0}^{\prime
}))_{1\leq k\leq m}$. It follows that
\begin{equation}
\int_{Q_{T}}g^{\varepsilon}(\cdot,\cdot,u_{\varepsilon}^{\prime})\psi
_{0}dxdW\rightarrow\int_{Q_{T}}\widetilde{g}(u_{0}^{\prime})\psi_{0}dxdW\text{
in law.}\label{0.7}%
\end{equation}

\textit{Step 2}. We focus at this level on the convergence of the stochastic
integral
\[
J_{\varepsilon}=\int_{Q_{T}}g^{\varepsilon}(\cdot,\cdot,u_{\varepsilon
}^{\prime})\psi_{0}dxdW^{\varepsilon}.
\]
We have
\begin{align*}
J_{\varepsilon} &  =\int_{Q_{T}}(g^{\varepsilon}(\cdot,\cdot,u_{\varepsilon
}^{\prime})-g^{\varepsilon}(\cdot,\cdot,u_{0}^{\prime}))\psi_{0}%
dxdW^{\varepsilon}+\int_{Q_{T}}g^{\varepsilon}(\cdot,\cdot,u_{0}^{\prime}%
)\psi_{0}dxdW^{\varepsilon}\\
&  =J_{1,\varepsilon}+J_{2,\varepsilon}.
\end{align*}
We proceed as in the Step 1 to show that $J_{1,\varepsilon}\rightarrow0$ as
$E\ni\varepsilon\rightarrow0$. Regarding $J_{2,\varepsilon}$, we need to show
that, as $E\ni\varepsilon\rightarrow0$,
\begin{equation}
\int_{Q_{T}}g^{\varepsilon}(\cdot,\cdot,u_{0}^{\prime})\psi_{0}%
dxdW^{\varepsilon}\rightarrow\int_{Q_{T}}\widetilde{g}(u_{0}^{\prime})\psi
_{0}dxdW\ \ \ \ \mathbb{P}\text{-a.s.}\label{0.8}%
\end{equation}
We consider each term separately as in Step 1, that is, we need to show that
\begin{equation}
\int_{Q_{T}}g_{k}^{\varepsilon}(\cdot,\cdot,u_{0}^{\prime})\psi_{0}%
dxdW_{k}^{\varepsilon}\rightarrow\int_{Q_{T}}\widetilde{g}(u_{0}^{\prime}%
)\psi_{0}dxdW_{k}\ \ \ \ \mathbb{P}\text{-a.s.}\label{0.9}%
\end{equation}
We first observe that the function
\[
g_{k}(\cdot,\cdot,u):(x,t,y,\tau,\omega)\mapsto g_{k}(y,\tau,u(x,t,\omega))
\]
belongs to $B(\Omega;\mathcal{C}(\overline{Q}_{T};B_{A}^{2,\infty}%
(\mathbb{R}_{y,\tau}^{N+1})))$ for $u\in B(\Omega;\mathcal{C}(\overline{Q}%
_{T}))$, and the latter space is dense in $L^{2}(Q_{T}\times\Omega)$. So it is
sufficient to check (\ref{0.9}) by replacing $g_{k}(\cdot,\cdot,u_{0}^{\prime
})$ by any element of $B(\Omega;\mathcal{C}(\overline{Q}_{T};B_{A}^{2,\infty
}(\mathbb{R}_{y,\tau}^{N+1})))$. However, as $\psi_{0}$ lies in $\mathcal{C}%
_{0}^{\infty}(Q_{T})$, it suffices to replace $g_{k}(\cdot,\cdot,u_{0}%
^{\prime})\psi_{0}$ by an element of $B(\Omega;\mathcal{K}(Q_{T}%
;B_{A}^{2,\infty}(\mathbb{R}_{y,\tau}^{N+1})))$. But, as in \cite[Lemma 3.1
and Proposition 3.3]{Ng1} (see also \cite[Proposition 4.5]{Hom1}) where it has
been shown (using a density argument) that we may replace the space
$L^{2}(Q_{T};A)$ (and so $\mathcal{K}(Q_{T};A)$) by $\mathcal{K}(Q_{T}%
;B_{A}^{2,\infty}(\mathbb{R}_{y,\tau}^{N+1}))$ in the definition of the
$\Sigma$-convergence, we can proceed in the same way to show that showing
(\ref{0.9})\ reduces in proving (using another density argument) that, as
$E\ni\varepsilon\rightarrow0$,
\begin{equation}
\int_{Q_{T}}\psi\left(  x,t,\frac{x}{\varepsilon},\frac{t}{\varepsilon
}\right)  dxdW_{k}^{\varepsilon}\rightarrow\int_{Q_{T}}M(\psi(x,t,\cdot
,\cdot))dxdW_{k}\rightarrow0\ \mathbb{P}\text{-a.s.}\label{0.10}%
\end{equation}
for any $\psi(x,t,y,\tau)=\varphi(x)\phi(y)\chi(t)\theta(\tau)$ with
$\varphi\in\mathcal{K}(Q)$, $\phi\in A_{y}$, $\chi\in\mathcal{K}(0,T)$ and
$\theta\in A_{\tau}$. But for $\psi$ as above, we have
\[
\int_{Q_{T}}\psi^{\varepsilon}dxdW_{k}^{\varepsilon}=\int_{Q}\varphi
(x)\phi\left(  \frac{x}{\varepsilon}\right)  dx\int_{0}^{T}\chi(t)\theta
\left(  \frac{t}{\varepsilon}\right)  dW_{k}^{\varepsilon}.
\]
At this level we consider the process
\[
M_{\varepsilon}(t)=\int_{0}^{t}\chi(s)\theta\left(  \frac{s}{\varepsilon
}\right)  dW_{k}^{\varepsilon}(s).
\]
We know that, setting $\Phi(t,\tau)=\chi(t)\theta(\tau)$ so that
$\Phi^{\varepsilon}(t)=\Phi(t,t/\varepsilon)$, the sequence of processes
$(\Phi^{\varepsilon},W_{k}^{\varepsilon})$ converges in law to $(M_{\tau}%
(\Phi(t,\cdot),W_{k})$ in $S_{0}=L^{2}(0,T)\times\mathcal{C}([0,T])$. So
following \cite[Proposition 2]{Lejay}, we need to show that the sequence
$(M_{\varepsilon})_{\varepsilon>0}$ is a \textit{good sequence} (see Appendix
A2 for the definition and characterization of good sequences). Indeed, in view
of Theorem \ref{tB1} associated to Definition \ref{dB2}, we have that the
quadratic variation $\left\langle M_{\varepsilon},M_{\varepsilon}\right\rangle
(t)$ of $M_{\varepsilon}$ is bounded independently of $\varepsilon$; indeed
\[
\left\langle M_{\varepsilon},M_{\varepsilon}\right\rangle (t)=\int_{0}%
^{t}\left\vert \Phi^{\varepsilon}(s)\right\vert ^{2}ds\leq t\left\Vert
\Phi\right\Vert _{\infty}^{2},
\]
so that the sequence $(M_{\varepsilon})_{\varepsilon>0}$ satisfies the
condition UCV (see Definition \ref{dB2}) and is hence a good sequence of
semimartingales. It follows readily that
\[
\int_{0}^{T}\Phi^{\varepsilon}(t)dW_{k}^{\varepsilon}(t)\rightarrow\int
_{0}^{T}M_{\tau}(\Phi(t,\cdot))dW_{k}(t).
\]
Using the convergence result
\[
\int_{Q}\varphi(x)\phi\left(  \frac{x}{\varepsilon}\right)  dx\rightarrow
\int_{Q}\varphi(x)M_{y}(\phi)dx,
\]
we readily get
\[
\int_{Q_{T}}\psi^{\varepsilon}dxdW_{k}^{\varepsilon}\rightarrow\int_{Q_{T}%
}M(\psi(x,t,\cdot,\cdot))dxdW_{k}.
\]
This completes the Step 2.

\textit{Step 3}. Putting together the results obtained in the previous steps,
we are led at once at (\ref{0.5}).
\end{proof}

Now, we use the previous convergence results (see especially (\textbf{C2})) to
deduce that the sequence $(u_{\varepsilon_{n}})_{n}$ is bounded in
$L^{2}(0,T;H_{0}^{1}(Q))$ $\overline{\mathbb{P}}$-a.s. Hence there exist a
subsequence of $(u_{\varepsilon_{n}})_{n}$ (not relabeled) which converge
weakly in $L^{2}(0,T;H_{0}^{1}(Q))$ to $u_{0}$ determined by (\textbf{C2}). It
follows from (SC)$_{3}$ and Lemma \ref{l2.1} associated to the Remark
\ref{r2.4} that there exists $u_{1}\in L^{2}(Q_{T};\mathcal{B}_{\#A_{y}}%
^{1,2}(\mathbb{R}_{y}^{N}))$ such that
\begin{equation}
\frac{\partial u_{\varepsilon_{n}}}{\partial x_{i}}\rightarrow\frac{\partial
u_{0}}{\partial x_{i}}+\frac{\overline{\partial}u_{1}}{\partial y_{i}}\text{
in }L^{2}(Q_{T})\text{-weak }\Sigma\text{ }(1\leq i\leq N)\text{ }%
\overline{\mathbb{P}}\text{-a.s.} \label{eqn14}%
\end{equation}
when $\varepsilon_{n}\rightarrow0$.

So we have that $(u_{0},u_{1})\in\mathbb{F}_{0}^{1}$ where
\[
\mathbb{F}_{0}^{1}=L^{2}(0,T;H_{0}^{1}(Q))\times L^{2}(Q_{T};\mathcal{B}%
_{\#A_{y}}^{1,2}(\mathbb{R}_{y}^{N})).
\]
It is an easy exercise in showing that the space $\mathcal{F}_{0}^{\infty
}=\mathcal{C}_{0}^{\infty}(Q_{T})\times(\mathcal{C}_{0}^{\infty}(Q_{T}%
)\otimes\mathcal{E})$ (where $\mathcal{E}=\varrho_{y}(A_{y}^{\infty
})/\mathbb{R}$, $\varrho_{y}$ being the canonical mapping of $B_{A}%
^{2}(\mathbb{R}_{y}^{N})$ into $\mathcal{B}_{A}^{2}(\mathbb{R}_{y}^{N})$
defined by $\varrho_{y}\left(  u\right)  =u+\mathcal{N}$ for $u\in B_{A}%
^{2}(\mathbb{R}_{y}^{N})$) is a dense subspace of $\mathbb{F}_{0}^{1}$.

For $\mathbf{v}=\left(  v_{0},v_{1}\right)  \in\mathbb{F}_{0}^{1}$, we set
$\mathbb{D}_{i}\mathbf{v=}\frac{\partial v_{0}}{\partial x_{i}}+\frac
{\overline{\partial}v_{1}}{\partial y_{i}}$ and $\mathbb{D}\mathbf{v}=\left(
\mathbb{D}_{i}\mathbf{v}\right)  _{1\leq i\leq N}\equiv\nabla v_{0}%
+\overline{\nabla}_{y}v_{1}$. We define $\mathbb{D}\Phi$ for $\Phi=(\psi
_{0},\psi_{1})\in\mathcal{F}_{0}^{\infty}$, mutatis mutandis.

With this in mind, we consider the following linear functional:
\[
q(\mathbf{u},\mathbf{v})=%
{\displaystyle\sum\limits_{i,j=1}^{N}}
\int_{Q_{T}}M\left(  a_{ij}\mathbb{D}_{j}\mathbf{u}\mathbb{D}_{i}%
\mathbf{v}\right)  dxdt
\]
for $\mathbf{u}=(u_{0},u_{1})\in\mathbb{F}_{0}^{1}$ and $\mathbf{v}=\left(
v_{0},v_{1}\right)  \in\mathbb{F}^{1}$. We also set $\widetilde{f}%
(u_{0}^{\prime})=M(f(\cdot,\cdot,u_{0}^{\prime}))$. Then the following
\textit{global} homogenization result holds.

\begin{proposition}
\label{p2.1} The couple $\mathbf{u}=\left(  u_{0},u_{1}\right)  \in
\mathbb{F}_{0}^{1}$ determined above solves the following variational problem:%
\begin{equation}
\left\{
\begin{array}
[c]{l}%
-\int_{Q_{T}}u_{0}^{\prime}\psi_{0}^{\prime}dxdt+q(u,\Phi)+\int_{Q_{T}%
}M\left(  \mathbb{D}\mathbf{u}^{\prime}\cdot\mathbb{D}\Phi\right)  dxdt\\
=\int_{Q_{T}}\widetilde{f}(u_{0}^{\prime})\psi_{0}dxdt+\int_{Q_{T}}%
\widetilde{g}(u_{0}^{\prime})\psi_{0}dxd\overline{W}\text{ \ }\overline
{\mathbb{P}}\text{-a.s.}\\
\text{for all }\Phi=(\psi_{0},\varrho_{y}(\psi))\in\mathcal{F}_{0}^{\infty}.
\end{array}
\right.  \label{eqn15}%
\end{equation}

\end{proposition}

\begin{proof}
In what follows we will write $\varepsilon$ instead of $\varepsilon_{n}$.
Thanks to Lemma \ref{l2.1}, we can choose the test functions under the form%
\[
\Phi_{\varepsilon}\left(  x,t\right)  =\psi_{0}\left(  x,t\right)
+\varepsilon\psi_{1}\left(  x,t,\frac{x}{\varepsilon}\right)  \ \text{\ }%
\left(  \left(  x,t\right)  \in Q_{T}\right)
\]
where $\psi_{0}\in\mathcal{C}_{0}^{\infty}\left(  Q_{T}\right)  $ and
$\psi_{1}\in\mathcal{C}_{0}^{\infty}\left(  Q_{T}\right)  \otimes
A_{y}^{\infty}$. Then using $\Phi_{\varepsilon}$ as a test function in the
variational formulation of (\ref{eqnn1.1}), we have%
\begin{equation}
\left\{
\begin{array}
[c]{c}%
-\int_{Q_{T}}u_{\varepsilon}^{\prime}\frac{\partial\Phi_{\varepsilon}%
}{\partial t}dxdt+\int_{Q_{T}}a^{\varepsilon}\nabla u_{\varepsilon}\cdot
\nabla\Phi_{\varepsilon}dxdt+\int_{Q_{T}}\nabla u_{\varepsilon}^{\prime}%
\cdot\nabla\Phi_{\varepsilon}dxdt\\
=\int_{Q_{T}}f^{\varepsilon}(\cdot,\cdot,u_{\varepsilon}^{\prime}%
)\Phi_{\varepsilon}dxdt+\int_{Q_{T}}g^{\varepsilon}(\cdot,\cdot,u_{\varepsilon
}^{\prime})\Phi_{\varepsilon}dxdW^{\varepsilon}.
\end{array}
\right.  \label{eqn33}%
\end{equation}
Our aim is to pass to the limit in (\ref{eqn33}). First, we have
\begin{align*}
\frac{\partial\Phi_{\varepsilon}}{\partial t}  &  =\frac{\partial\psi_{0}%
}{\partial t}+\varepsilon\left(  \frac{\partial\psi_{1}}{\partial t}\right)
^{\varepsilon},\\
\nabla\Phi_{\varepsilon}  &  =\nabla\psi_{0}+\varepsilon\left(  \nabla\psi
_{1}\right)  ^{\varepsilon}+\left(  \nabla_{y}\psi_{1}\right)  ^{\varepsilon
}\text{,}\\
\nabla\Phi_{\varepsilon}^{\prime}  &  =\nabla\psi_{0}^{\prime}+\varepsilon
\left(  \nabla\psi_{1}^{\prime}\right)  ^{\varepsilon}+\left(  \nabla_{y}%
\psi_{1}^{\prime}\right)  ^{\varepsilon}.
\end{align*}
Now, we use the usual $\Sigma$-convergence method (see Remark \ref{r2.3}) to
obtain, as $\varepsilon\rightarrow0$,
\begin{equation}
\frac{\partial\Phi_{\varepsilon}}{\partial t}\rightarrow\frac{\partial\psi
_{0}}{\partial t}\text{ in }L^{2}\left(  0,T;H^{-1}\left(  Q\right)  \right)
\text{-weak} \label{eqn34}%
\end{equation}%
\begin{equation}
\nabla\Phi_{\varepsilon}\rightarrow\nabla\psi_{0}+\nabla_{y}\psi_{1}\text{ in
}L^{2}(Q_{T})^{N}\text{-strong }\Sigma\label{eqn35}%
\end{equation}

\begin{equation}
\nabla\Phi_{\varepsilon}^{\prime}\rightarrow\nabla\psi_{0}^{\prime}+\nabla
_{y}\psi_{1}^{\prime}\text{ in }L^{2}(Q_{T})^{N}\text{-strong }\Sigma.
\label{4.1}%
\end{equation}%
\begin{equation}
\Phi_{\varepsilon}\rightarrow\psi_{0}\text{ in }L^{2}(Q_{T})\text{-strong. }
\label{eqn36}%
\end{equation}
One can easily show, using (\ref{eqn14}) and (\ref{eqn35}), that
\[
\nabla u_{\varepsilon}\cdot\nabla\Phi_{\varepsilon}\rightarrow\mathbb{D}%
\mathbf{u}\cdot\mathbb{D}\Phi\text{ in }L^{1}(Q_{T})\text{-weak }\Sigma\text{
\ }\overline{\mathbb{P}}\text{-a.s.}%
\]

Using the strong convergence of $u_{\varepsilon}^{\prime}$ towards
$u_{0}^{\prime}$ in $L^{2}(Q_{T})$ together with (\ref{eqn34}) we have
\begin{equation}
\int_{Q_{T}}u_{\varepsilon}^{\prime}\frac{\partial\Phi_{\varepsilon}}{\partial
t}dxdt\rightarrow\int_{Q_{T}}u_{0}^{\prime}\psi_{0}^{\prime}dxdt.
\label{eqn38}%
\end{equation}

Next using the fact that $a_{ij}\in\mathcal{C}(\overline{Q};B_{A_{y}%
}^{2,\infty}(\mathbb{R}_{y}^{N}))\subset\mathcal{C}(\overline{Q}_{T};B_{A_{y}%
}^{2,\infty}(\mathbb{R}_{y}^{N}))$, we get
\begin{equation}
\int_{Q_{T}}a^{\varepsilon}\nabla u_{\varepsilon}\cdot\nabla\Phi_{\varepsilon
}dxdt\rightarrow q(u,\Phi)\text{.} \label{eqn37}%
\end{equation}
Considering the next term we have%
\[
\int_{Q_{T}}\nabla u_{\varepsilon}^{\prime}\cdot\nabla\Phi_{\varepsilon
}dxdt=\int_{Q_{T}}\nabla u_{\varepsilon}\cdot\nabla\Phi_{\varepsilon}^{\prime
}dxdt,
\]

and using (\ref{eqn14}) we obtain
\[
\int_{Q_{T}}\nabla u_{\varepsilon}^{\prime}\cdot\nabla\Phi_{\varepsilon
}dxdt\rightarrow\int_{Q_{T}}M\left(  \left(  \nabla u_{0}+\overline{\nabla
}_{y}u_{1}\right)  \left(  \nabla\psi_{0}^{\prime}+\nabla_{y}\psi_{1}^{\prime
}\right)  \right)  dxdt.
\]
However
\[
\int_{Q_{T}}M\left(  \left(  \nabla u_{0}+\overline{\nabla}_{y}u_{1}\right)
\left(  \nabla\psi_{0}^{\prime}+\nabla_{y}\psi_{1}^{\prime}\right)  \right)
dxdt=\int_{Q_{T}}M(\mathbb{D}\mathbf{u}^{\prime}\cdot\mathbb{D}\Phi)dxdt,
\]
so that
\begin{equation}
\int_{Q_{T}}\nabla u_{\varepsilon}^{\prime}\cdot\nabla\Phi_{\varepsilon
}dxdt\rightarrow\int_{Q_{T}}M\left(  \mathbb{D}\mathbf{u}^{\prime}%
\cdot\mathbb{D}\Phi\right)  dxdt. \label{eqn40a}%
\end{equation}

Let us now deal with the stochastic term. We have
\begin{align*}
\int_{Q_{T}}g^{\varepsilon}(\cdot,\cdot,u_{\varepsilon}^{\prime}%
)\Phi_{\varepsilon}dxdW^{\varepsilon} &  =\int_{Q_{T}}g^{\varepsilon}%
(\cdot,\cdot,u_{\varepsilon}^{\prime})\psi_{0}dxdW^{\varepsilon}\\
&  +\varepsilon\int_{Q_{T}}g^{\varepsilon}(\cdot,\cdot,u_{\varepsilon}%
^{\prime})\psi_{1}^{\varepsilon}dxdW^{\varepsilon}\\
&  =I_{1}+I_{2}.
\end{align*}
Appealing to Lemma \ref{l3.2}, we get that
\[
I_{1}\rightarrow\int_{Q_{T}}\widetilde{g}(u_{0}^{\prime})\psi_{0}dxdW\text{ in
law.}%
\]
Concerning $I_{2}$, we proceed as in \cite{7'} (using the
Burkh\"{o}lder-Davis-Gundy's inequality) to show that $I_{2}\rightarrow0$.
Thus
\begin{equation}
\int_{Q_{T}}g^{\varepsilon}(\cdot,\cdot,u_{\varepsilon}^{\prime}%
)\Phi_{\varepsilon}dxdW^{\varepsilon}\rightarrow\int_{Q_{T}}\widetilde
{g}(u_{0}^{\prime})\psi_{0}dxdW\text{ in law.}\label{eqn40b}%
\end{equation}
Finally we use (\ref{eqn42'}) in Lemma \ref{l4.1} to get
\begin{equation}
\int_{Q_{T}}f^{\varepsilon}(\cdot,\cdot,u_{\varepsilon}^{\prime}%
)\Phi_{\varepsilon}dxdt\rightarrow\int_{Q_{T}}\widetilde{f}(u_{0}^{\prime
})\psi_{0}dxdt.\label{eqn40c}%
\end{equation}

Putting together (\ref{eqn38}), (\ref{eqn37}), (\ref{eqn40a}), (\ref{eqn40b})
and (\ref{eqn40c}) we obtain the result.
\end{proof}

The variational problem (\ref{eqn15}) is the global homogenized problem for
(\ref{eqnn1.1}).

\subsection{Homogenized problem}

In order to derive the homogenized problem we need to deal with an equivalent
expression of problem (\ref{eqn15}). As we can see, this problem is equivalent
to the following system (\ref{eqn16})-(\ref{eqn17}) reading as
\begin{equation}
\int_{Q_{T}}M\left[  a\mathbb{D}\mathbf{u+}\mathbb{D}\mathbf{u}^{\prime
}\right]  \cdot\nabla_{y}\psi_{1}dxdt=0\text{, for all }\psi_{1}\in
\mathcal{C}_{0}^{\infty}\left(  Q_{T}\right)  \otimes A_{y}^{\infty}
\label{eqn16}%
\end{equation}

\begin{equation}
\left\{
\begin{array}
[c]{l}%
-\int_{Q_{T}}u_{0}^{\prime}\psi_{0}^{\prime}dxdt+q(u,(\psi_{0},0))+\int
_{Q_{T}}M\left(  \mathbb{D}\mathbf{u}^{\prime}\cdot\nabla\psi_{0}\right)
dxdt\\
=\int_{Q_{T}}\widetilde{f}(u_{0}^{\prime})\psi_{0}dxdt+\int_{Q_{T}}%
\widetilde{g}(u_{0}^{\prime})\psi_{0}dxd\overline{W}\\
\text{for all }\in\psi_{0}\in\mathcal{C}_{0}^{\infty}(Q_{T}).
\end{array}
\right.  \label{eqn17}%
\end{equation}
We first deal with (\ref{eqn16}). Choosing
\begin{equation}
\psi_{1}(x,t,y,\tau)=\varphi(x,t)\phi(y)\chi(\tau) \label{eqn44}%
\end{equation}
with $\varphi\in\mathcal{C}_{0}^{\infty}(Q_{T})$ and $\phi\in A_{y}^{\infty}$,
it follows that%

\begin{equation}
M_{y}\left(  \left(  a\mathbb{D}\mathbf{u+}\mathbb{D}\mathbf{u}^{\prime
}\right)  \cdot\nabla_{y}\phi\right)  =0\text{,}\ \text{for all }\phi\in
A_{y}^{\infty}, \label{eqn45a}%
\end{equation}
and since $M_{y}(\mathbb{D}\mathbf{u}^{\prime}\cdot\nabla_{y}\phi)=0$,
(\ref{eqn16}) finally becomes
\begin{equation}
M_{y}\left(  \left(  a\mathbb{D}\mathbf{u}\right)  \cdot\nabla_{y}\phi\right)
=0\text{,}\ \text{for all }\phi\in A_{y}^{\infty}, \label{4.2}%
\end{equation}

So for $\xi\in\mathbb{R}^{N}$ be freely fixed, we consider the cell problem:
\begin{equation}
\left\{
\begin{array}
[c]{c}%
\text{Find }\pi\left(  \xi\right)  \in\mathcal{B}_{\#A_{y}}^{1,2}%
(\mathbb{R}_{y}^{N})\text{ such that :}\\
-\overline{\operatorname{div}}_{y}\left(  a(x,\cdot)((\xi+\overline{\nabla
}_{y}\pi(\xi))\right)  =0\ \text{in }\mathbb{R}_{y}^{N}.
\end{array}
\right.  \label{eqn46}%
\end{equation}
Due to the properties of $a$, the cell problem (\ref{eqn46}) admits a unique solution.

Now choosing $\xi=\nabla u_{0}(x,t)$ in (\ref{eqn46}) and testing the
resulting equation with $\psi$ as in (\ref{eqn44}), we get by uniqueness of
solution of (\ref{eqn46}) that $u_{1}(x,t,y)=\pi(\nabla u_{0}(x,t))(y)$ \ for
a.e $\left(  x,t\right)  \in Q_{T}$.\ From which the uniqueness of $u_{1}$
defined as above and belonging to $L^{2}(Q_{T};\mathcal{B}_{\#A_{y}}%
^{1,2}(\mathbb{R}_{y}^{N}))$ and it is easy to see that
\begin{equation}
u_{1}(x,t,y)=\nabla u_{0}(x,t)\cdot\chi(x,y), \label{*}%
\end{equation}
where $\chi(x,\cdot)=(\chi_{j}(x,\cdot))_{1\leq j\leq N}$ with $\chi
_{j}(x,\cdot)=\pi(e_{j})$, $e_{j}$ being the $j$ th vector of the canonical
basis in $\mathbb{R}^{N}$.

We define the homogenized coefficients as follows:

\begin{itemize}
\item The homogenized operator associated to $P^{\varepsilon}$ $(0<\varepsilon
<1)$ is $\widetilde{P}$ and defined by
\[
\widetilde{P}=-\sum_{i,j=1}^{N}\frac{\partial}{\partial x_{i}}\left(  M\left(
a_{ij}(x,\cdot)(I_{N}+\nabla_{y}\chi(x,\cdot)))\right)  \frac{\partial
}{\partial x_{j}}\right)
\]
where $I_{N}$ is the $N\times N$ identity matrix.

\item The homogenized functions $\widetilde{f}(u_{0}^{\prime})$ and
$\widetilde{g}(u_{0}^{\prime})$ as given previously.
\end{itemize}

\begin{remark}
\label{r3.0}\emph{It is an easy exercise in showing that the functions
}$\widetilde{f}$\emph{\ and }$\widetilde{g}$\emph{\ (as functions of the
argument }$u_{0}^{\prime}\in\mathbb{R}$\emph{) satisfy assumptions similar to
those of }$f$\emph{\ and }$g$\emph{, and that }$\widetilde{P}$\emph{\ is
uniformly elliptic.}
\end{remark}

The following result is provides us with the equivalent (or homogenized) model
for which $u_{0}$ is the solution.

\begin{proposition}
\label{p2.2}The function $u_{0}$ is solution of the boundary value problem:%
\begin{equation}
\left\{
\begin{array}
[c]{c}%
du_{0}^{\prime}+\widetilde{P}u_{0}dt-\Delta u_{0}^{\prime}dt=\widetilde
{f}(u_{0}^{\prime})dt+\widetilde{g}(u_{0}^{\prime})d\overline{W}\text{ \ in
}Q_{T}\\
u_{0}=0\text{ on }\partial Q_{T}\times\left(  0,T\right)  \text{
\ \ \ \ \ \ \ \ \ \ \ \ \ \ \ \ \ \ \ \ \ \ \ \ \ \ \ \ \ \ \ }\\
u_{0}(x,0)=u^{0}\text{ and }u_{0}^{\prime}(x,0)=u^{1}\text{ in }Q.\text{
\ \ \ \ \ \ \ \ \ \ \ \ }%
\end{array}
\right.  \label{eqn}%
\end{equation}

\end{proposition}

\begin{proof}
Substituting in (\ref{eqn17}) $u_{1}$ by its expression given by (\ref{*}) and
choosing there $\psi_{0}(x,t,\omega)=\varphi(x,t)\phi(\omega)$, with
$\varphi\in\mathcal{C}_{0}^{\infty}(Q_{T})$ and $\phi\in B(\overline{\Omega}%
)$, we get readily (\ref{eqn}).
\end{proof}

The next result establishes the uniqueness of the solution of (\ref{eqn}) on
the same probability system.

\begin{proposition}
\label{p2.3}Let $u_{0}$ and $u_{0}^{\ast}$ be two solutions of
\emph{(\ref{eqn})} on the same probability system $(\overline{\Omega
},\overline{\mathcal{F}},\overline{\mathbb{P}},\overline{W},\overline
{\mathcal{F}^{t}})$ with the same initial conditions $u^{0}$ and $u^{1}$. Then
$u_{0}=u_{0}^{\ast}\ \overline{\mathbb{P}}$-almost surely.
\end{proposition}

\begin{proof}
Choosing\textbf{\ }$w_{0}=u_{0}-u_{0}^{\ast}$, we have%
\[
\left\{
\begin{array}
[c]{c}%
dw_{0}^{\prime}+\widetilde{P}u_{0}dt-\Delta w_{0}^{\prime}dt=(\widetilde
{f}(u_{0}^{\prime})-\widetilde{f}(u_{0}^{\ast\prime}))dt+(\widetilde{g}%
(u_{0}^{\prime})-\widetilde{g}(u_{0}^{\ast\prime}))d\overline{W}\\
w_{0}(x,0)=0,\text{ }w_{0}^{\prime}(x,0)=0.
\end{array}
\right.
\]
Applying Ito's formula and integrating over $[0,t]$ yield%

\begin{align}
\left\Vert w_{0}(t)\right\Vert _{H_{0}^{1}\left(  Q\right)  }^{2}+\left\Vert
w_{0}^{\prime}(t)\right\Vert _{L^{2}\left(  Q\right)  }^{2}+\int_{0}%
^{t}\left\Vert w_{0}^{\prime}(t)\right\Vert _{H_{0}^{1}\left(  Q\right)  }%
^{2}dt  &  \leq\int_{0}^{t}\left\Vert \widetilde{g}(u_{0}^{\prime
}(s))-\widetilde{g}(u_{0}^{\ast\prime}(s))\right\Vert _{L^{2}(Q)^{m}}%
^{2}ds\label{eqn49}\\
&  +2\int_{0}^{t}(\widetilde{f}(u_{0}^{\prime}(s)-\widetilde{f}(u_{0}%
^{\ast\prime}(s)),w_{0}^{\prime}(s))ds\nonumber\\
&  +2\int_{0}^{t}(\widetilde{g}(u_{0}^{\prime}(s))-\widetilde{g}(u_{0}%
^{\ast\prime}(s)),w_{0}^{\prime}(s))d\overline{W}\left(  s\right)  .\nonumber
\end{align}

Hence%
\begin{align}
2\int_{0}^{t}(\widetilde{f}(u_{0}^{\prime}(s))-\widetilde{f}(u_{0}^{\ast
\prime}(s)),w_{0}^{\prime}(s))ds  &  \leq C\int_{0}^{t}\left(  \left\Vert
\widetilde{f}(u_{0}^{\prime}(s))-\widetilde{f}(u_{0}^{\ast\prime
}(s))\right\Vert _{L^{2}(Q)}^{2}+\left\Vert w_{0}^{\prime}(s)\right\Vert
_{L^{2}(Q)}^{2}\right)  ds\nonumber\\
&  \leq C\int_{0}^{t}\left(  \left\Vert w_{0}(s)\right\Vert _{H_{0}^{1}%
(Q)}^{2}+\left\Vert w_{0}^{\prime}(s)\right\Vert _{L^{2}(Q)}^{2}\right)  ds.
\label{eqn50'}%
\end{align}

It follows that%
\begin{equation}
\int_{0}^{t}\left\Vert \widetilde{g}(u_{0}^{\prime}(s))-\widetilde{g}%
(u_{0}^{\ast\prime}(s))\right\Vert _{L^{2}(Q)^{m}}^{2}ds\leq C\int_{0}%
^{t}\left(  \left\Vert w_{0}(s)\right\Vert _{H_{0}^{1}(Q)}^{2}+\left\Vert
w_{0}^{\prime}(s)\right\Vert _{L^{2}(Q)}^{2}\right)  ds \label{eqn51}%
\end{equation}
Taking the mathematical expectation in (\ref{eqn49}), we appeal to the
assumptions on $g$ and combine (\ref{eqn50'}) with (\ref{eqn51}) to get%

\[
\overline{\mathbb{E}}\left(  \left\Vert w_{0}(t)\right\Vert _{H_{0}^{1}\left(
Q\right)  }^{2}+\left\Vert w_{0}^{\prime}(t)\right\Vert _{L^{2}\left(
Q\right)  }^{2}\right)  \leq C\overline{\mathbb{E}}\int_{0}^{t}(\left\Vert
w_{0}(s)\right\Vert _{H_{0}^{1}(Q)}^{2}+\left\Vert w_{0}^{\prime
}(s)\right\Vert _{L^{2}(Q)}^{2})ds\text{.}%
\]
It emerges from Gronwall's lemma that $w_{0}=0\ \overline{\mathbb{P}}$-almost surely.
\end{proof}

\begin{remark}
\emph{The pathwise uniqueness result in proposition \ref{p2.3} and
Yamada-Watanabe's Theorem \cite{9} yield the existence of unique strong
probabilistic solution of (\ref{eqn}) on a prescribed probabilistic system
}$(\Omega,\mathcal{F},\mathbb{P}$\emph{,}$W,\mathcal{F}^{t})$\emph{.}
\end{remark}

The aim of the rest of this section is to prove the following homogenization\ result.

\begin{theorem}
\label{t2.6}For each $\varepsilon>0$ let $u_{\varepsilon}$ be the unique
solution of \emph{(\ref{eqnn1.1})} on a given stochastic system $\left(
\Omega,\mathcal{F},\mathbb{P},W,\mathcal{F}^{t}\right)  $ defined as in
Section \emph{1}. Under the assumptions (\textbf{A1}), (\textbf{A2}) and
(\textbf{A3}), the sequence $(u_{\varepsilon})$ converges in probability to
$u_{0}$ in $L^{2}\left(  Q_{T}\right)  $, where $u_{0}$ is the unique strong
probabilistic solution of the following problem:%
\begin{equation}
\left\{
\begin{array}
[c]{c}%
du_{0}^{\prime}+\widetilde{P}u_{0}dt-\Delta u_{0}^{\prime}dt=\widetilde
{f}(u_{0}^{\prime})dt+\widetilde{g}(u_{0}^{\prime})dW\text{ in }Q_{T}\\
u_{0}=0\text{ on }\partial Q\times\left(  0,T\right)  \text{
\ \ \ \ \ \ \ \ \ \ \ \ \ \ \ \ \ \ \ \ \ \ \ \ \ \ \ \ \ \ }\\
u_{0}\left(  x,0\right)  =u^{0}\text{ and }u_{0}^{\prime}\left(  x,0\right)
=u^{1}\text{ in }Q.\text{ \ \ \ \ \ \ \ \ \ }%
\end{array}
\right.  \label{eqn43}%
\end{equation}

\end{theorem}

The proof of this theorem uses the pathwise uniqueness for (\ref{eqn}) in
conjunction with the following two lemmas.

\begin{lemma}
[\cite{6}]\label{l4.2}Let $X$ be a polish space equipped with its Borel
$\sigma$-algebra. A sequence of $X$-valued random variables $\left\{
Y_{n},\text{ }n\in\mathbb{N}\right\}  $ converges in probability if and only
if for every joint laws $\left\{  \mu_{n_{k}},~k\in\mathbb{N}\right\}  $,
there exist a further subsequence which converges weakly to a probability
measure $\mu$ such that
\[
\mu((x,y)\in X\times X:x=y)=1.
\]

\end{lemma}

Consider $X=L^{2}(0,T;L^{2}(Q))\cap\mathcal{C}(0,T;H^{-1}(Q))\times
\mathcal{C}(0,T;H^{-1}(Q))$, $X_{1}=\mathcal{C}(0,T;\mathbb{R}^{m})$,
$X_{2}=X\times X\times X_{1}$. For $S\in\mathcal{B}(X)$, we set $\pi
^{\varepsilon}(S)=\mathbb{P}((u_{\varepsilon},u_{\varepsilon}^{\prime})\in
S)$. For $S\in\mathcal{B}(X_{1})$, we set $\pi^{W}=\mathbb{P}(W\in S) $.

We define the joint probability laws as%
\[
\pi^{\varepsilon,\varepsilon^{\prime}}=\pi^{\varepsilon}\times\pi
^{\varepsilon^{\prime}},
\]%
\[
\nu^{\varepsilon,\varepsilon^{\prime}}=\pi^{\varepsilon}\times\pi
^{\varepsilon^{\prime}}\times\pi^{W}.
\]

\begin{lemma}
\label{l4.3} The family $\{\nu^{\varepsilon,\varepsilon^{\prime}}%
:\varepsilon,\varepsilon^{\prime}>0\}$ is tight on $(X_{2},\mathcal{B}(X_{2}))
$.
\end{lemma}

\begin{proof}
The proof of this lemma is similar to that of \cite[Lemma 7]{7'}.
\end{proof}

We now have all the ingredients to prove Theorem \ref{t2.6}.

\begin{proof}
[Proof of Theorem \ref{t2.6}]Thanks to Lemma \ref{l4.3} and Skorokhod's
theorem \cite{1} we infer the existence of a subsequence from $\{\nu
^{\varepsilon_{j},\varepsilon_{j}^{\prime}}\}$ still denoted by $\{\nu
^{\varepsilon_{j},\varepsilon_{j}^{\prime}}:\varepsilon,\varepsilon^{\prime
}>0\}$ which converges to a probability measure $\nu$ on $(X_{2}%
,\mathcal{B}(X_{2}))$ and there exists a probability space $(\overline{\Omega
},\overline{\mathcal{F}},\overline{\mathbb{P}})$ on which the sequence
$((u_{\varepsilon_{j}},u_{\varepsilon_{j}}^{\prime}),(u_{\varepsilon
_{j}^{\prime}},u_{\varepsilon_{j}^{\prime}}^{\prime}),W^{j})$ is defined and
converges almost surely in $X_{2}$ to a couple of random variables
$((u_{0},u_{0}^{\prime}),(v_{0},v_{0}^{\prime}),\overline{W})$. We also have
\[
\mathcal{L}((u_{\varepsilon_{j}},u_{\varepsilon_{j}}^{\prime}),(u_{\varepsilon
_{j}^{\prime}},u_{\varepsilon_{j}^{\prime}}^{\prime}),W^{j})=\nu
^{\varepsilon_{j},\varepsilon_{j}^{\prime}}%
\]%
\[
\mathcal{L}((u_{0},u_{0}^{\prime}),(v_{0},v_{0}^{\prime}),\overline{W})=\nu.
\]
We set
\begin{align*}
Z_{j}^{u_{\varepsilon},u_{\varepsilon}^{\prime}}  &  =((u_{\varepsilon_{j}%
},u_{\varepsilon_{j}}^{\prime}),W^{j})\text{ and }Z_{j}^{u_{\varepsilon_{j}%
},u_{\varepsilon_{j}}^{\prime}}=((u_{\varepsilon_{j}^{\prime}},u_{\varepsilon
_{j}^{\prime}}^{\prime}),W^{j}),\\
Z^{(u_{0},u_{0}^{\prime})}  &  =((u_{0},u_{0}^{\prime}),\overline{W})\text{
and }Z^{(v_{0},v_{0}^{\prime})}=((v_{0},v_{0}^{\prime}),\overline{W})\text{.}%
\end{align*}
It follows from the above argument that $\pi^{\varepsilon_{j},\varepsilon
_{j}^{\prime}}$ converges to a measure $\pi$ such that%
\[
\pi(\cdot)=\overline{\mathbb{P}}((u_{0},u_{0}^{\prime}),(v_{0},v_{0}^{\prime
})\in\cdot).
\]
It is an easy task to see that $Z_{j}^{u_{\varepsilon},u_{\varepsilon}%
^{\prime}}$ and $Z_{j}^{u_{\varepsilon_{j}},u_{\varepsilon_{j}}^{\prime}}$
satisfy (\ref{eqn32}) and that $Z^{\left(  u_{0},u_{0}^{\prime}\right)  }$ and
$Z^{\left(  v_{0},v_{0}^{\prime}\right)  }$ satisfy (\ref{eqnn1.1}) on the
same stochastic system $(\overline{\Omega},\overline{\mathcal{F}}%
,\overline{\mathbb{P}},\overline{W},\overline{\mathcal{F}^{t}})$ where
$\overline{\mathcal{F}^{t}}$ is the filtration generated by the couple
$((u_{0},u_{0}^{\prime}),(v_{0},v_{0}^{\prime}),\overline{W})$. We deduce from
the uniqueness result (see Proposition \ref{p2.3}) that $u_{0}=v_{0}$ in
$L^{2}(Q_{T}),u_{0}^{\prime}=v_{0}^{\prime}$ in $L^{2}(0,T;H^{-1}(Q))$.

Hence we have%
\begin{align}
&  \pi((x,y),(x^{\prime},y^{\prime})\in X\times X:(x,y)=(x^{\prime},y^{\prime
}))\label{eqn48}\\
&  =\overline{\mathbb{P}}((u_{0},u_{0}^{\prime})=(v_{0},v_{0}^{\prime})\text{
in }X)\nonumber\\
&  =1.\nonumber
\end{align}
\ Thanks to (\ref{eqn48}) and Lemma \ref{l4.3} we conclude that the original
sequence $(u_{\varepsilon},u_{\varepsilon}^{\prime})$ defined on the original
probability space $(\Omega,\mathcal{F},\mathbb{P},W,\mathcal{F}^{t}) $
converges in probability to $(u_{0},u_{0}^{\prime})$ in the topology of $X$,
which implies that the sequence $(u_{\varepsilon})$ converges in probability
to $u_{0}$ in $L^{2}(Q_{T})$ and $(u_{\varepsilon}^{\prime})$ converges in
probability to $u_{0}^{\prime}$ in $L^{2}(0,T;H^{-1}(Q))$. By the passage to
the limit as in the previous subsection we can show that$\ u_{0}$ is the
unique strong solution of (\ref{eqn43}).
\end{proof}

\section{Some applications of Theorem\textbf{\ }\ref{t2.6}}

We have made a fundamental assumption (\textbf{A3}) under which the multiscale
analysis of (\ref{eqnn1.1}) has been made possible. Here we give some concrete
situations in which this holds true.

\begin{problem}
[\textbf{Periodic setting}]\label{prob5.3}\emph{We assume here that
}$a(x,\cdot)$\emph{, }$f(\cdot,\cdot,v)$\emph{\ and }$g(\cdot,\cdot
,v)$\emph{\ are periodic with period }$1$\emph{\ in each coordinate for each
}$x\in\overline{Q}$\emph{\ and }$v\in\mathbb{R}$\emph{. The appropriate
algebras with mean value here are }$A_{y}=\mathcal{C}_{per}(Y)$\emph{\ and
}$A_{\tau}=\mathcal{C}_{per}(\mathcal{T})$\emph{\ and so }$A=\mathcal{C}%
_{per}(Y\times\mathcal{T})$\emph{\ where }$Y=(0,1)^{N}$\emph{\ and
}$\mathcal{T}=(0,1)$\emph{. Where }$\mathcal{C}_{per}(Y)$ $($\emph{resp
}$\mathcal{C}_{per}(\mathcal{T})$\emph{\ and }$\mathcal{C}_{per}%
(Y\times\mathcal{T}))$\emph{\ is the Banach algebra of continuous }%
$Y$\emph{-periodic functions defined on }$\mathbb{R}^{N}$\emph{\ (resp on
}$\mathbb{R}$\emph{\ and }$\mathbb{R}^{N+1}$\emph{). We therefore have
}$B_{A_{y}}^{p}(\mathbb{R}^{N})=L_{per}^{p}(Y)$\emph{, }$B_{A}^{p}%
(\mathbb{R}^{N+1})=L_{per}^{p}(Y\times\mathcal{T})$\emph{\ for }$1\leq
p\leq\infty$\emph{\ and }$B_{\#A_{y}}^{1,2}(\mathbb{R}^{N})=H_{\#}%
^{1}(Y)=\left\{  u\in H_{per}^{1}(Y):\int_{Y}udy=0\right\}  $\emph{. The
homogenized coefficients are defined as follows.}%
\begin{align*}
\widetilde{P}  &  =-\sum_{i,j=1}^{N}\frac{\partial}{\partial x_{i}}\left(
\int_{Y}\left(  a_{ij}(x,\cdot)(I_{N}+\nabla_{y}\chi(x,\cdot)))\right)
\frac{\partial}{\partial x_{j}}\right) \\
\widetilde{f}(v)  &  =\iint_{Y\times\mathcal{T}}f(y,\tau,v)dyd\tau\text{
\emph{and} }\widetilde{g}(v)=\left(  \iint_{Y\times\mathcal{T}}\left\vert
g(y,\tau,v)\right\vert ^{2}dyd\tau\right)  ^{\frac{1}{2}}%
\end{align*}
\emph{where }$I_{N}$\emph{\ is the }$N\times N$\emph{\ identity matrix, and
}$\chi=(\chi_{j})_{1\leq j\leq N}$\emph{\ is the solution of the cell problem
}%
\[
\left\{
\begin{array}
[c]{l}%
\chi_{j}(x,\cdot)\in H_{\#}^{1}(Y):\\
-\operatorname{div}_{y}(a(x,\cdot)(e_{j}+\nabla_{y}\chi_{j}(x,\cdot
)))=0\text{\emph{\ in }}Y.
\end{array}
\right.
\]
\emph{This problem has been considered in \cite{7'}, but with a linear
operator not involving the damping term.}
\end{problem}

\begin{problem}
[\textbf{Almost periodic setting}]\label{prob5.1}\emph{We assume that the
function }$a(x,\cdot),$\emph{\ }$f(\cdot,\cdot,v)$\emph{\ and }$g(\cdot
,\cdot,v)$\emph{\ are Besicovitch almost periodic \cite{2''} , for any }%
$x\in\overline{Q}$ \emph{and} $v\in\mathbb{R}$\emph{. Then we perform the
homogenization of (\ref{eqnn1.1}) with }$A_{y}=AP(\mathbb{R}^{N})$\emph{,
}$A_{\tau}=AP(\mathbb{R})$\emph{\ and so }$A=AP(\mathbb{R}^{N+1}%
)$\emph{\ where }$AP(\mathbb{R}^{N})$\emph{\ \cite{2}\ is the algebra of Bohr
almost periodic functions on }$\mathbb{R}^{N}$\emph{. Let us remind that the
mean value of a function }$u\in AP(\mathbb{R}^{N})$\emph{\ is the unique
constant belonging to the close convex hull of the family of the translates
}$(u(\cdot+a))_{a\in\mathbb{R}^{N}}$\emph{\ of }$u$\emph{.}
\end{problem}

\begin{problem}
\label{prob5.2}\emph{We consider the space }$\mathcal{B}_{\infty}%
(\mathbb{R};X)$\emph{\ of all continuous functions }$\psi\in\mathcal{C}\left(
\mathbb{R};X\right)  $\emph{\ such that }$\psi(\zeta)$\emph{\ has a limit in
}$X$\emph{\ as }$\left\vert \zeta\right\vert \rightarrow+\infty$\emph{, where
}$X$\emph{\ is a Banach space}.\emph{\ The space }$\mathcal{B}_{\infty}\left(
\mathbb{R}^{N};\mathbb{R}\right)  =\mathcal{B}_{\infty}\left(  \mathbb{R}%
^{N}\right)  $\emph{\ is an algebra with mean value on }$\mathbb{R}^{N}%
$\emph{. Then we perform the homogenization of (\ref{eqnn1.1}) under the
assumption that:} $a(x,\cdot)\in L_{per}^{2}(Y)^{N\times N}$\emph{\ for any
}$x\in\overline{Q}$\emph{, }$f(\cdot,\cdot,v)$\emph{, }$g(\cdot,\cdot,v)\in
B_{\infty}(\mathbb{R}_{\tau};L_{per}^{2}(Y))$\emph{\ for all }$v\in
\mathbb{R},$ \emph{where }$A_{y}=\mathcal{C}_{per}(Y)$\emph{\ and }$A_{\tau
}=\mathcal{B}_{\infty}(\mathbb{R}_{\tau}),$ \emph{hence}\ $A=\mathcal{B}%
_{\infty}(\mathbb{R}_{\tau})\odot\mathcal{C}_{per}(Y)=\mathcal{B}_{\infty
}(\mathbb{R}_{\tau};\mathcal{C}_{per}(Y))$\emph{.}
\end{problem}

\section{Appendix A1}

\begin{lemma}
\label{l2.5}Under the assumptions (\textbf{A1})-(\textbf{A2}), the solution
$u_{\varepsilon}$ of the problem \emph{(\ref{eqnn1.1})} satisfies the
following estimates%
\begin{equation}
\mathbb{E}\sup_{0\leq t\leq T}\left\Vert u_{\varepsilon}(t)\right\Vert
_{H_{0}^{1}\left(  Q\right)  }^{2}+\mathbb{E}\sup_{0\leq t\leq T}\left\Vert
u_{\varepsilon}^{\prime}(t)\right\Vert _{L^{2}\left(  Q\right)  }%
^{2}+\mathbb{E}\int_{0}^{T}\left\Vert u_{\varepsilon}^{\prime}(t)\right\Vert
_{H_{0}^{1}\left(  Q\right)  }^{2}dt\leq C\text{.} \label{eqn10aa}%
\end{equation}%
\begin{equation}
\text{ }\mathbb{E}\sup_{0\leq t\leq T}\left\Vert u_{\varepsilon}(t)\right\Vert
_{L^{2}\left(  Q\right)  }^{2}\text{, }\mathbb{E}\sup_{0\leq t\leq
T}\left\Vert u_{\varepsilon}(t)\right\Vert _{H_{0}^{1}\left(  Q\right)  }%
^{4}+\mathbb{E}\sup_{0\leq t\leq T}\left\Vert u_{\varepsilon}^{\prime
}(t)\right\Vert _{L^{2}\left(  Q\right)  }^{4}\leq C\text{.} \label{eqn11}%
\end{equation}
where $C$ is a positive constant independent of $\varepsilon$.
\end{lemma}

\begin{proof}
By It\^{o}'s formula and integrating over $\left[  0,t\right]  $ with $0\leq
t\leq T$,we obtain
\begin{align}
&  \left\Vert u_{\varepsilon}^{\prime}(t)\right\Vert _{L^{2}\left(  Q\right)
}^{2}+2\int_{0}^{t}\left(  P^{\varepsilon}u_{\varepsilon}(s),u_{\varepsilon
}^{\prime}(s)\right)  ds-2\int_{0}^{t}(\Delta u_{\varepsilon}^{\prime
}(s),u_{\varepsilon}^{\prime}(s))ds=\left\Vert u^{1}\right\Vert _{L^{2}\left(
Q\right)  }^{2}+\label{eqn19a}\\
&  +2\int_{0}^{t}\left(  f\left(  \frac{x}{\varepsilon},\frac{s}{\varepsilon
},u_{\varepsilon}^{\prime}(s)\right)  ,u_{\varepsilon}^{\prime}\left(
s\right)  \right)  ds+2\int_{0}^{t}\left(  g\left(  \frac{x}{\varepsilon
},\frac{s}{\varepsilon},u_{\varepsilon}^{\prime}\left(  s\right)  \right)
,u_{\varepsilon}^{\prime}\left(  s\right)  \right)  dW(s)\nonumber\\
&  +\int_{0}^{t}\left\Vert g\left(  \frac{x}{\varepsilon},\frac{s}%
{\varepsilon},u_{\varepsilon}^{\prime}\left(  s\right)  \right)  \right\Vert
_{L^{2}\left(  Q\right)  ^{m}}^{2}ds.\nonumber
\end{align}
So we obtain%
\begin{align}
&  \left\Vert u_{\varepsilon}(t)\right\Vert _{H_{0}^{1}\left(  Q\right)  }%
^{2}+\left\Vert u_{\varepsilon}^{\prime}(t)\right\Vert _{L^{2}\left(
Q\right)  }^{2}+2\int_{0}^{t}\left\Vert u_{\varepsilon}^{\prime}(s)\right\Vert
_{H_{0}^{1}\left(  Q\right)  }^{2}ds\label{eqna}\\
&  \leq C\left(  \left\Vert u^{1}\right\Vert _{L^{2}\left(  Q\right)  }%
^{2}+\left\Vert u^{0}\right\Vert _{H_{0}^{1}\left(  Q\right)  }^{2}\right)
+2\int_{0}^{t}\left(  f\left(  \frac{x}{\varepsilon},\frac{s}{\varepsilon
},u_{\varepsilon}^{\prime}\left(  s\right)  \right)  ,u_{\varepsilon}^{\prime
}\left(  s\right)  \right)  ds\nonumber\\
&  \mathbb{+}\int_{0}^{t}\left\Vert g\left(  \frac{x}{\varepsilon},\frac
{s}{\varepsilon},u_{\varepsilon}^{\prime}\left(  s\right)  \right)
\right\Vert _{L^{2}\left(  Q\right)  ^{m}}^{2}ds+2\int_{0}^{t}\left(  g\left(
\frac{x}{\varepsilon},\frac{s}{\varepsilon},u_{\varepsilon}^{\prime}\left(
s\right)  \right)  ,u_{\varepsilon}^{\prime}\left(  s\right)  \right)
dW\left(  s\right) \nonumber
\end{align}
Taking the supremum in $t\in\left[  0,T\right]  $, and the mathematical
expectation in both side of (\ref{eqna}), and using also the assumption
(\textbf{A1}) we have
\begin{align*}
&  \mathbb{E}\sup_{0\leq t\leq T}\left(  \left\Vert u_{\varepsilon
}(t)\right\Vert _{H_{0}^{1}\left(  Q\right)  }^{2}+\left\Vert u_{\varepsilon
}^{\prime}(t)\right\Vert _{L^{2}\left(  Q\right)  }^{2}\right)  +2\mathbb{E}%
\int_{0}^{T}\left\Vert u_{\varepsilon}^{\prime}(s)\right\Vert _{H_{0}%
^{1}\left(  Q\right)  }^{2}dt\\
&  \leq C\left(  \left\Vert u^{1}\right\Vert _{L^{2}\left(  Q\right)  }%
^{2}+\left\Vert u^{0}\right\Vert _{H_{0}^{1}\left(  Q\right)  }^{2}\right)
+2\mathbb{E}\int_{0}^{T}\left(  f\left(  \frac{x}{\varepsilon},\frac
{s}{\varepsilon},u_{\varepsilon}^{\prime}\left(  s\right)  \right)
,u_{\varepsilon}^{\prime}\left(  s\right)  \right)  ds\\
&  \mathbb{E}\int_{0}^{T}\left\Vert g\left(  \frac{x}{\varepsilon},\frac
{s}{\varepsilon},u_{\varepsilon}^{\prime}\left(  s\right)  \right)
\right\Vert _{L^{2}\left(  Q\right)  ^{m}}^{2}ds+2\sup_{0\leq t\leq
T}\mathbb{E}\int_{0}^{t}\left(  g\left(  \frac{x}{\varepsilon},\frac
{s}{\varepsilon},u_{\varepsilon}^{\prime}\left(  s\right)  \right)
,u_{\varepsilon}^{\prime}\left(  s\right)  \right)  dW\left(  s\right)  .
\end{align*}
Making use of H\"{o}lder's inequality, Young's inequality and the assumption
on $f$, we obtain
\begin{align}
&  2\mathbb{E}\int_{0}^{T}\left(  f\left(  \frac{x}{\varepsilon},\frac
{s}{\varepsilon},u_{\varepsilon}^{\prime}\left(  s\right)  \right)
,u_{\varepsilon}^{\prime}\left(  s\right)  \right)  ds\nonumber\\
&  \leq\mathbb{E}\int_{0}^{T}\left(  \left\Vert f\left(  \frac{x}{\varepsilon
},\frac{s}{\varepsilon},u_{\varepsilon}^{\prime}\left(  s\right)  \right)
\right\Vert _{L^{2}\left(  Q\right)  }^{2}+\left\Vert u_{\varepsilon}^{\prime
}\left(  s\right)  \right\Vert _{L^{2}\left(  Q\right)  }^{2}\right)
ds\nonumber\\
&  \leq C\mathbb{E}\int_{0}^{T}\left(  1+\left\Vert u_{\varepsilon}\left(
s\right)  \right\Vert _{H_{0}^{1}\left(  Q\right)  }^{2}+\left\Vert
u_{\varepsilon}^{\prime}\left(  s\right)  \right\Vert _{L^{2}\left(  Q\right)
}^{2}\right)  ds.
\end{align}
Using the assumption on $g$, we get%

\begin{align}
&  \mathbb{E}\int_{0}^{T}\left\Vert g\left(  \frac{x}{\varepsilon},\frac
{s}{\varepsilon},u_{\varepsilon}^{\prime}\left(  s\right)  \right)
\right\Vert _{L^{2}\left(  Q\right)  ^{m}}^{2}ds\label{eqn20a}\\
&  \leq C\int_{0}^{T}\left(  1+\left\Vert u_{\varepsilon}\left(  s\right)
\right\Vert _{H_{0}^{1}\left(  Q\right)  }^{2}+\left\Vert u_{\varepsilon
}^{\prime}\left(  s\right)  \right\Vert _{L^{2}\left(  Q\right)  }^{2}\right)
ds.\nonumber
\end{align}

Applying Burkh\"{o}lder-Davis-Gundy's inequality, we get%
\begin{align}
&  \mathbb{E}\sup_{0\leq t\leq T}\int_{0}^{t}\left(  g\left(  \frac
{x}{\varepsilon},\frac{s}{\varepsilon},u_{\varepsilon}^{\prime}\left(
s\right)  \right)  ,u_{\varepsilon}^{\prime}\left(  s\right)  \right)
dW\left(  s\right) \nonumber\\
&  \leq C\mathbb{E}\left(  \int_{0}^{T}\left(  g\left(  \frac{x}{\varepsilon
},\frac{s}{\varepsilon},u_{\varepsilon}^{\prime}\left(  s\right)  \right)
,u_{\varepsilon}^{\prime}\left(  s\right)  \right)  ^{2}ds\right)  ^{\frac
{1}{2}}\nonumber\\
&  \leq C\mathbb{E}\left(  \int_{0}^{T}\left\Vert g\left(  \frac
{x}{\varepsilon},\frac{s}{\varepsilon},u_{\varepsilon}^{\prime}\left(
s\right)  \right)  \right\Vert _{L^{2}\left(  Q\right)  ^{m}}^{2}\left\Vert
u_{\varepsilon}^{\prime}\left(  s\right)  \right\Vert _{L^{2}\left(  Q\right)
}^{2}ds\right)  ^{\frac{1}{2}}\nonumber\\
&  \leq C\mathbb{E}\left(  \int_{0}^{T}\left(  1+\left\Vert u_{\varepsilon
}\left(  s\right)  \right\Vert _{H_{0}^{1}\left(  Q\right)  }^{2}+\left\Vert
u_{\varepsilon}^{\prime}\left(  s\right)  \right\Vert _{L^{2}\left(  Q\right)
}^{2}\right)  \left\Vert u_{\varepsilon}^{\prime}\left(  s\right)  \right\Vert
_{L^{2}\left(  Q\right)  }^{2}ds\right)  ^{\frac{1}{2}}\\
&  \leq C\mathbb{E}\sup_{0\leq t\leq T}\left(  1+\left\Vert u_{\varepsilon
}\left(  s\right)  \right\Vert _{H_{0}^{1}\left(  Q\right)  }^{2}+\left\Vert
u_{\varepsilon}^{\prime}\left(  s\right)  \right\Vert _{L^{2}\left(  Q\right)
}^{2}\right)  ^{\frac{1}{2}}\left(  \int_{0}^{T}\left\Vert u_{\varepsilon
}^{\prime}\left(  s\right)  \right\Vert _{L^{2}\left(  Q\right)  }%
^{2}ds\right)  ^{\frac{1}{2}}\nonumber\\
&  \leq\frac{1}{2}C\mathbb{E}\sup_{0\leq t\leq T}\left(  1+\left\Vert
u_{\varepsilon}\left(  s\right)  \right\Vert _{H_{0}^{1}\left(  Q\right)
}^{2}+\left\Vert u_{\varepsilon}^{\prime}\left(  s\right)  \right\Vert
_{L^{2}\left(  Q\right)  }^{2}\right) \label{eqn21}\\
&  +\frac{1}{2}C\mathbb{E}\int_{0}^{T}\left\Vert u_{\varepsilon}^{\prime
}\left(  s\right)  \right\Vert _{L^{2}\left(  Q\right)  }^{2}ds\nonumber
\end{align}
Combining (\ref{eqn20a}), (\ref{eqn21}), we obtain%
\begin{align*}
&  \mathbb{E}\sup_{0\leq t\leq T}\left\Vert u_{\varepsilon}\left(  t\right)
\right\Vert _{H_{0}^{1}\left(  Q\right)  }^{2}+\mathbb{E}\sup_{0\leq t\leq
T}\left\Vert u_{\varepsilon}^{\prime}\left(  t\right)  \right\Vert
_{L^{2}\left(  Q\right)  }^{2}+2\mathbb{E}\int_{0}^{T}\left\Vert
u_{\varepsilon}^{\prime}\left(  t\right)  \right\Vert _{H_{0}^{1}\left(
Q\right)  }^{2}dt\\
&  \leq C\left(  \left\Vert u^{1}\right\Vert _{L^{2}\left(  Q\right)  }%
^{2}+\left\Vert u^{0}\right\Vert _{H_{0}^{1}\left(  Q\right)  }^{2}\right)
+C(1+\left\Vert u_{\varepsilon}\left(  s\right)  \right\Vert _{H_{0}%
^{1}\left(  Q\right)  }^{2}+\left\Vert u_{\varepsilon}^{\prime}\left(
s\right)  \right\Vert _{L^{2}\left(  Q\right)  }^{2})\\
&  +C\mathbb{E}\int_{0}^{T}\left\Vert u_{\varepsilon}^{\prime}\left(
s\right)  \right\Vert _{L^{2}\left(  Q\right)  }^{2}ds.
\end{align*}
The Gronwall's inequality then gives%

\[
\mathbb{E}\sup_{0\leq t\leq T}\left\Vert u_{\varepsilon}\left(  t\right)
\right\Vert _{H_{0}^{1}\left(  Q\right)  }^{2}+\mathbb{E}\sup_{0\leq t\leq
T}\left\Vert u_{\varepsilon}^{\prime}\left(  t\right)  \right\Vert
_{L^{2}\left(  Q\right)  }^{2}+\mathbb{E}\int_{0}^{T}\left\Vert u_{\varepsilon
}^{\prime}(t)\right\Vert _{H_{0}^{1}\left(  Q\right)  }^{2}dt\leq C.
\]
We can now prove the last estimate. Taking the square in both side of the
inequality (\ref{eqna}) we have%
\begin{align}
&  \left\Vert u_{\varepsilon}(t)\right\Vert _{H_{0}^{1}\left(  Q\right)  }%
^{4}+\left\Vert u_{\varepsilon}^{\prime}(t)\right\Vert _{L^{2}\left(
Q\right)  }^{4}\leq C\left(  \left\Vert u^{1}\right\Vert _{L^{2}\left(
Q\right)  }^{4}+\left\Vert u^{0}\right\Vert _{H_{0}^{1}\left(  Q\right)  }%
^{4}\right)  \text{ }\label{eqn22}\\
&  +C\left(  \int_{0}^{t}\left(  f\left(  \frac{x}{\varepsilon},\frac
{s}{\varepsilon},u_{\varepsilon}^{\prime}\left(  s\right)  \right)
,u_{\varepsilon}^{\prime}\left(  s\right)  \right)  ds\right)  ^{2}\text{
}+4\left(  \int_{0}^{t}\left(  g\left(  \frac{x}{\varepsilon},\frac
{s}{\varepsilon},u_{\varepsilon}^{\prime}\left(  s\right)  \right)
,u_{\varepsilon}^{\prime}\left(  s\right)  \right)  dW\left(  s\right)
\right)  ^{2}.\nonumber
\end{align}
Taking the supremum with respect to $t\in\left[  0,T\right]  $ and using the
assumption on $f$, we get
\begin{align}
\mathbb{E}\left(  \int_{0}^{T}\left(  f\left(  \frac{x}{\varepsilon},\frac
{s}{\varepsilon},u_{\varepsilon}^{\prime}\left(  s\right)  \right)
,u_{\varepsilon}^{\prime}\left(  s\right)  \right)  ds\right)  ^{2}  &  \leq
C\mathbb{E}\left(  \int_{0}^{T}\left(  \left\Vert f\left(  \frac
{x}{\varepsilon},\frac{s}{\varepsilon},u_{\varepsilon}^{\prime}\left(
s\right)  \right)  \right\Vert _{L^{2}\left(  Q\right)  }^{2}+\left\Vert
u_{\varepsilon}^{\prime}\left(  s\right)  \right\Vert _{L^{2}\left(  Q\right)
}^{2}\right)  ds\right)  ^{2}\nonumber\\
&  \leq C\mathbb{E}\int_{0}^{T}\left(  1+\left\Vert u_{\varepsilon}\left(
s\right)  \right\Vert _{H_{0}^{1}\left(  Q\right)  }^{4}+\left\Vert
u_{\varepsilon}^{\prime}\left(  s\right)  \right\Vert _{L^{2}\left(  Q\right)
}^{4}\right)  ds. \label{eqn23a}%
\end{align}
Next, using again the Burkh\"{o}lder-Davis-Gundy's inequality, we obtain
\begin{align}
\mathbb{E}\sup_{0\leq t\leq T}\left(  \int_{0}^{t}\left(  g\left(  \frac
{x}{\varepsilon},\frac{s}{\varepsilon},u_{\varepsilon}^{\prime}\left(
s\right)  \right)  ,u_{\varepsilon}^{\prime}\left(  s\right)  \right)
dW\left(  s\right)  \right)  ^{2}  &  \leq\mathbb{E}\left(  \int_{0}%
^{T}\left(  g\left(  \frac{x}{\varepsilon},\frac{s}{\varepsilon}%
,u_{\varepsilon}^{\prime}\left(  s\right)  \right)  ,u_{\varepsilon}^{\prime
}\left(  s\right)  \right)  ^{2}ds\right) \label{eqn24}\\
&  \leq C\mathbb{E}\int_{0}^{T}\left(  1+\left\Vert u_{\varepsilon}\left(
s\right)  \right\Vert _{H_{0}^{1}\left(  Q\right)  }^{4}+\left\Vert
u_{\varepsilon}^{\prime}\left(  s\right)  \right\Vert _{L^{2}\left(  Q\right)
}^{4}\right)  ds.\nonumber
\end{align}
Collecting the results (\ref{eqn22}), (\ref{eqn23a}), (\ref{eqn24}) and using
Gronwall inequality we obtain
\[
\mathbb{E}\sup_{0\leq t\leq T}\left\Vert u_{\varepsilon}(t)\right\Vert
_{H_{0}^{1}\left(  Q\right)  }^{4}+\mathbb{E}\sup_{0\leq t\leq T}\left\Vert
u_{\varepsilon}^{\prime}(t)\right\Vert _{L^{2}\left(  Q\right)  }^{4}\leq C.
\]
This ends the proof.
\end{proof}

\begin{lemma}
\label{l2.6}Under the assumptions of Lemma \ref{l2.5}, we have
\begin{equation}
\mathbb{E}\sup_{\left\vert \theta\right\vert \leq\delta}\int_{0}^{T}\left\Vert
u_{\varepsilon}^{\prime}(t+\theta)-u_{\varepsilon}^{\prime}(t)\right\Vert
_{H^{-1}(Q)}^{2}dt\leq C\delta\text{.} \label{eqn12}%
\end{equation}
where $\delta>0$ is a small parameter, $C$ a positive constant independent of
$\varepsilon$ and $\delta$.
\end{lemma}

\begin{proof}
From (\ref{eqnn1.1}) we can write, for$\ \theta\geq0$,
\begin{align*}
u_{\varepsilon}^{\prime}\left(  t+\theta\right)  -u_{\varepsilon}^{\prime
}\left(  t\right)   &  =\int_{t}^{t+\theta}P^{\varepsilon}u_{\varepsilon
}ds+\int_{t}^{t+\theta}\Delta u_{\varepsilon}^{\prime}ds+\int_{t}^{t+\theta
}f\left(  \frac{x}{\varepsilon},\frac{s}{\varepsilon},u_{\varepsilon}^{\prime
}\left(  s\right)  \right)  ds\\
&  +\int_{t}^{t+\theta}g\left(  \frac{x}{\varepsilon},\frac{s}{\varepsilon
},u_{\varepsilon}^{\prime}\left(  s\right)  \right)  dW\left(  s\right)  .
\end{align*}
Then%
\begin{align}
\left\Vert u_{\varepsilon}^{\prime}\left(  t+\theta\right)  -u_{\varepsilon
}^{\prime}\left(  t\right)  \right\Vert _{H^{-1}\left(  Q\right)  }^{2}  &
\leq\left(  \int_{t}^{t+\theta}\left\Vert u_{\varepsilon}\left(  s\right)
\right\Vert _{H_{0}^{1}(Q)}ds\right)  ^{2}+\left(  \int_{t}^{t+\theta
}\left\Vert \nabla u_{\varepsilon}^{\prime}\left(  s\right)  \right\Vert
_{L^{2}(Q)}ds\right)  ^{2}\label{eqn25}\\
&  +C\left(  \int_{t}^{t+\theta}\left\Vert f\left(  \frac{x}{\varepsilon
},\frac{s}{\varepsilon},u_{\varepsilon}^{\prime}\left(  s\right)  \right)
\right\Vert _{L^{2}(Q)}ds\right)  ^{2}+\left\vert \int_{t}^{t+\theta}g\left(
\frac{x}{\varepsilon},\frac{s}{\varepsilon},u_{\varepsilon}^{\prime}\left(
s\right)  \right)  dW\left(  s\right)  \right\vert ^{2}.\text{ }\nonumber
\end{align}
Integrating between $0$ and $T$ and taking the mathematical expectation, we
have%
\begin{align}
\mathbb{E}\sup_{0\leq\theta\leq\delta}\int_{0}^{T-\theta}\left\Vert
u_{\varepsilon}^{\prime}\left(  t+\theta\right)  -u_{\varepsilon}^{\prime
}\left(  t\right)  \right\Vert _{H^{-1}\left(  Q\right)  }^{2}dt  &  \leq
C\delta^{2}\left(  \mathbb{E}\sup_{0\leq s\leq T}\left\Vert u_{\varepsilon
}\left(  s\right)  \right\Vert _{H_{0}^{1}(Q)}^{2}+\mathbb{E}\int_{0}%
^{T}\left\Vert \nabla u_{\varepsilon}^{\prime}\left(  s\right)  \right\Vert
_{L^{2}(Q)}^{2}ds\right) \nonumber\\
&  C\delta^{2}\left(  1+\mathbb{E}\sup_{0\leq s\leq T}\left(  \left\Vert
u_{\varepsilon}\left(  s\right)  \right\Vert _{H_{0}^{1}(Q)}^{2}+\left\Vert
u_{\varepsilon}^{\prime}\left(  s\right)  \right\Vert _{L^{2}(Q)}^{2}\right)
\right) \nonumber\\
&  +\mathbb{E}\int_{0}^{T-\theta}\sup_{0\leq\theta\leq\delta}\left\vert
\int_{t}^{t+\theta}g\left(  \frac{x}{\varepsilon},\frac{s}{\varepsilon
},u_{\varepsilon}^{\prime}\left(  s\right)  \right)  dW\left(  s\right)
\right\vert ^{2}dt \label{eqn26}%
\end{align}
Next using Burkh\"{o}lder-Davis-Gundy's inequality we obtain
\begin{align}
&  \mathbb{E}\int_{0}^{T-\theta}\sup_{0\leq\theta\leq\delta}\left\vert
\int_{t}^{t+\theta}g\left(  \frac{x}{\varepsilon},\frac{s}{\varepsilon
},u_{\varepsilon}^{\prime}\left(  s\right)  \right)  dW\left(  s\right)
\right\vert ^{2}dt\nonumber\\
&  \leq\mathbb{E}\int_{0}^{T}\left(  \int_{t}^{t+\theta}\left\Vert g\left(
\frac{x}{\varepsilon},\frac{s}{\varepsilon},u_{\varepsilon}^{\prime}\left(
s\right)  \right)  \right\Vert ^{2}ds\right)  dt\nonumber\\
&  \leq C\mathbb{E}\int_{0}^{T}\left(  \int_{t}^{t+\theta}(1+\left\Vert
u_{\varepsilon}\left(  s\right)  \right\Vert _{H_{0}^{1}(Q)}^{2}+\left\Vert
u_{\varepsilon}^{\prime}\left(  s\right)  \right\Vert _{L^{2}(Q)}%
^{2})ds\right)  dt\nonumber\\
&  \leq C\delta, \label{eqn27}%
\end{align}
we have used the assumption on $g$. Collecting the results and making the same
reasoning with $\theta<0$, we finally obtain%
\[
\mathbb{E}\sup_{\left\vert \theta\right\vert \leq\delta}\int_{0}^{T}\left\Vert
u_{\varepsilon}^{\prime}\left(  t+\theta\right)  -u_{\varepsilon}^{\prime
}\left(  t\right)  \right\Vert _{H^{-1}\left(  Q\right)  }^{2}dt\leq
C\delta\text{,}%
\]
whence the lemma.
\end{proof}

\section{Appendix A2}

Let $(\Omega,\mathcal{A},\mathbb{P})$ be a given probability space with
expectation $\mathbb{E}$. We recall here some facts about the convergence of
the integral $\int_{0}^{t}H_{\varepsilon}dX_{\varepsilon}(s)$ where
$(X_{\varepsilon})_{\varepsilon>0}$ is a sequence of semimartingales and
$(H_{\varepsilon})_{\varepsilon>0}$ is a sequence of c\`{a}dl\`{a}g
(right-continuous with left limit) processes such that $H_{\varepsilon}$ is
adapted to the filtration generated by $X_{\varepsilon}$. Here we follow the
presentation of \cite{Lejay} which is borrowed from \cite{KP}. We begin with
some definitions borrowed from the preceding references.

\begin{definition}
[Good sequence]\label{dB1}\emph{A sequence of c\`{a}dl\`{a}g semimartigales
}$(X_{\varepsilon})_{\varepsilon>0}$\emph{ is said to be a }good
sequence\emph{ if }$(X_{\varepsilon})_{\varepsilon>0}$\emph{ converges in law
to a process }$X$\emph{ and for any sequence }$(H_{\varepsilon})_{\varepsilon
>0}$\emph{ of c\`{a}dl\`{a}g processes such that }$H_{\varepsilon}$\emph{ is
adapted to the filtration generated by }$X_{\varepsilon}$\emph{ and
}$(H_{\varepsilon},X_{\varepsilon})$\emph{ converges in law to }$(H,X)$\emph{,
then }$X$\emph{ is a semimartingale with respect to the smallest filtration
}$\mathcal{H}=(\mathcal{H}(t))_{t\geq0}$\emph{ generated by }$(H,X)$\emph{
satisfying the usual hypotheses, and, when all the involved stochastic
integrals are defined, }%
\[
\int_{0}^{t}H_{\varepsilon}(s)dX_{\varepsilon}(s)\rightarrow\int_{0}%
^{t}H(s)dX(s)\text{ }\mathbb{P}\text{-a.s.}%
\]

\end{definition}

From now on, we restrict ourselves to the sequence of semimartingales defined
by $X^{\varepsilon}(t)\equiv M^{\varepsilon}(t)=\int_{0}^{t}H_{\varepsilon
}(s)dB_{\varepsilon}(s)$ where $(B_{\varepsilon}(t))_{\varepsilon>0}$ is a
sequence of $1$-dimensional Browian motion defined on $(\Omega,\mathcal{A}%
,\mathbb{P})$.

\begin{definition}
[{Condition UCV; \cite[Definition 7.5, p. 23]{KP}}]\label{dB2}\emph{A sequence
of continuous semimartingales }$(X_{\varepsilon})_{\varepsilon>0}$\emph{ is
said to satisfy }condition \emph{UCV (that is, }$(X_{\varepsilon
})_{\varepsilon>0}$\emph{ is said to have }Uniformly Controlled
Variations\emph{)} \emph{if for each }$\eta>0$\emph{ and each }$\varepsilon
>0$\emph{, there exists a stopping time }$T^{\varepsilon,\eta}$\emph{ such
that }$\mathbb{P}(T^{\varepsilon,\eta}\leq\eta)\leq\frac{1}{\eta}$\emph{ and
the quadratic variation }$\left\langle M_{\varepsilon},M_{\varepsilon
}\right\rangle (t)$ of $M_{\varepsilon}$ is such that $\varepsilon$%
\begin{equation}
\sup_{\varepsilon>0}\mathbb{E}\left[  \left\langle M_{\varepsilon
},M_{\varepsilon}\right\rangle (\min(1,T^{\varepsilon,\eta}))\right]
\equiv\sup_{\varepsilon>0}\int_{0}^{\min(1,T^{\varepsilon,\eta})}%
\mathbb{E}\left\vert H_{\varepsilon}(s)\right\vert ^{2}ds<\infty. \label{C1}%
\end{equation}

\end{definition}

\begin{remark}
\label{rB1}\emph{The condition (\ref{C1}) in the above definition is given in
\cite{Lejay} for a more general sequence of processes as the one above. Here
we adapt it to our setting.}
\end{remark}

The next result provides us with a characterization of good sequences.

\begin{theorem}
[{\cite[Theorem 1]{Lejay}}]\label{tB1}Let $(X_{\varepsilon})_{\varepsilon>0}$
be a sequence of semimartigales converging in law to some process $X$. Then,
the sequence $(X_{\varepsilon})_{\varepsilon>0}$ is good if and only if it
satisfies the condition \emph{UCV}.
\end{theorem}

\begin{acknowledgement}
\emph{The authors acknowledge the support of the CETIC (African Center of
Excellence in Information and Communication Technologies).}
\end{acknowledgement}

\end{document}